\documentclass[11pt,a4paper]{article}
\usepackage{epsf,epsfig,amsfonts,amsgen,amsmath,amssymb,amstext,amsbsy,amsopn,amsthm,cases,listings,color
}

\usepackage[english]{babel}
\usepackage{booktabs}

\usepackage[letterpaper,top=2cm,bottom=2cm,left=3cm,right=3cm,marginparwidth=1.75cm]{geometry}

\usepackage{amsmath}
\usepackage{graphicx}
\usepackage[colorlinks=true, allcolors=blue]{hyperref}
\usepackage[capitalise]{cleveref}
\usepackage{verbatim}
\usepackage{subcaption}
\usepackage{float}
\usepackage{authblk}
\usepackage{tikz}
\usetikzlibrary{shapes.geometric, arrows.meta, positioning, calc}

\newtheorem{definition}{Definition} [section]
\newtheorem{theorem}[definition]{Theorem}
\newtheorem{lemma}[definition]{Lemma}
\newtheorem{proposition}[definition]{Proposition}
\newtheorem{remark}[definition]{Remark}
\newtheorem{corollary}[definition]{Corollary}

\newtheorem{claim}[definition]{Claim}
\newtheorem{problem}[definition]{Problem}
\newtheorem{observation}[definition]{Observation}
\newtheorem{fact}[definition]{Fact}
\newenvironment{mycase}[1]{%
    \par\addvspace{\medskipamount}%
    \noindent\textbf{#1}\quad%
}{%
    \par\addvspace{\medskipamount}%
}

\newenvironment{poc}[1][Proof]{%
  \begin{proof}[#1]%
  }{%
  \end{proof}%
}

\title{On the chromatic profile for tripartite graphs and beyond}
\date{}
\author[1]{Bo Ning\thanks{Email: \texttt{bo.ning@nankai.edu.cn}. Supported by National Natural Science Foundation of China, grant No. 12371350.}}
\author[2]{Jian Wang\thanks{Email: \texttt{wangjianmath01@scu.edu.cn}. Supported by  National Natural Science Foundation of China, grant No. 12471316.}}
\author[3]{Yisai Xue\thanks{Email: \texttt{xueyisai@nbu.edu.cn}. Supported by  National Natural Science Foundation of China, grant No. 12501486.}}
\affil[1]{{\small College of Computer Science, Nankai University, Tianjin, P.R. China.}}
\affil[2]{{\small Department of Mathematics, 
            Sichuan University, Chengdu, P.R. China.}}
\affil[3]{{\small  School of Mathematics and Statistics, Ningbo University, Ningbo, P.R. China.}}
\begin{document}
\maketitle

\begin{abstract}
Let $H$ be a graph  and let $\delta_{\chi}(H,r)$ denote the infimum of $c$ such that every $H$-free graph with minimum degree at least $cn$ is $r$-colorable. 
The \textit{chromatic profile} of $H$ is defined to be the values of $\delta_{\chi}(H,r)$ as $r$ varies.
Erd\H{o}s and Simonovits
described this graph parameter as ``too complicated", 
and Allen, B\"ottcher, Griffiths, Kohayakawa, and Morris posed its determination for every graph $H$ as an open problem \cite[Problem~45]{ABGKM2013}, emphasizing its expected difficulty.

In this paper, we resolve the case $r=2$ for every graph $H$ with $\chi(H)=3$. 
We show that the set of possible values of $\delta_{\chi}(H,2)$ with $\chi(H)=3$ is finite and discrete:
$$\{\delta_{\chi}(H,2):\chi(H)=3\}=\left\{\frac{1}{2},\frac{2}{5},\frac{2}{7},\frac{1}{4},\frac{2}{9},\frac{1}{5},\frac{2}{11},\frac{1}{6}\right\}.$$
 Furthermore, we provide a complete structural characterization of the graphs $H$ associated with each threshold value. 
Moreover, we extend the classical chromatic profile result for triangle to color-critical graphs $H$ with $g_{\mathrm{odd}}(H)=\chi(H)=3$.

Our approach introduces a useful auxiliary parameter.
Motivated by the notion of vertex-extendability of Liu, Mubayi, and Reiher \cite{liu2023unified},
we define the {\it vertex-extendable threshold} of $H$,
denoted by $\delta_{\mathrm{ext}}(H,r)$, as the infimum of $c\in (0,1)$ so that for every $H$-free graph $G$ on $n$ vertices, the existence of a vertex $v \in V(G)$ with $\chi(G - v) \leq r$ combined with $\delta(G)\ge cn$ implies that $G$ is $r$-colorable.
A key structural consequence is that
$\delta_{\chi}(H,2) = \max\left\{\delta_\chi(C_{2k+1},2),\delta_{\mathrm{ext}}(H,2)\right\},$
where $H$ is a color-critical graph with $\chi(H)=3$ and $g_{\mathrm{odd}}(H)=2k+1$ for $k\geq 2$. 
\end{abstract}

\section{Introduction}

Let $G$ and $H$ be graphs. We say that $G$ is \emph{$H$-free} if it contains no copy of $H$.
A central theme in extremal graph theory is to understand how forbidding a fixed graph $H$
forces global structure in a host graph $G$. 
A fundamental parameter in this area is
the \emph{Tur\'an number} ${\rm ex}(n,H)$, defined as the maximum number of edges in an
$H$-free graph on $n$ vertices.
  Mantel determined that ${\rm ex}(n,K_3)=\lfloor n^2/4\rfloor$,
and Tur\'an~\cite{TU41} extended this by determining ${\rm ex}(n,K_{r+1})$ for all $r\ge 2$.
Since then, Tur\'an-type problems have been a central topic in extremal graph theory.

For a graph $H$, the \emph{chromatic number of $H$}, denoted by $\chi(H)$, is the minimum number of colors required to color the vertices of $H$ so that any two adjacent vertices receive different colors. 
For any fixed graph $H$ with $\chi(H) \ge 3$, the celebrated Erd\H{o}s--Stone--Simonovits theorem \cite{erdos1966limit,erdos1946structure} characterizes the asymptotics of the extremal number as 
\begin{align}\label{thm-ESS}
    \mathrm{ex}(n, H) = \left(1-\frac{1}{\chi(H)-1}+o(1)\right) \binom{n}{2}.
\end{align}
Thus the extremal density is governed only by $\chi(H)$.  By contrast, if one asks for
minimum-degree conditions forcing an $H$-free graph to have small chromatic number, the answer
is much more sensitive to the structure of $H$.

We denote by
$\delta_{\chi}(H,r)$ the infimum of $c$ such that every $H$-free graph with minimum degree at least $c\cdot|V(G)|$ is $r$-colorable.  Precisely,
\begin{align*}
    \delta_\chi(H, r):=\inf \{c: \delta(G) \geq c\cdot|V(G)| \text { and } H \not \subseteq G \Rightarrow \chi(G) \leq r\}.
\end{align*}
The \textit{chromatic profile} of a graph $H$ is defined to be the collection of values $\delta_{\chi}(H,r)$.

One of the earliest results of this type is the theorem of Andr\'asfai, Erd\H{o}s and S\'os,
which implies $\delta_\chi(K_{r+1},r)=\frac{3r-4}{3r-1}$.

\begin{theorem}[Andr\'{a}sfai-Erd\H{o}s-S\'{o}s Theorem \cite{AES74}]\label{thm-AES}
Every $K_{r+1}$-free graph $G$ on $n$ vertices with minimum degree greater than $\frac{3r-4}{3r-1} n$ must be $r$-partite.
\end{theorem}

For triangle-free graphs, a sequence of works of H\"aggkvist \cite{haggkvist1982odd}, Jin \cite{jin1995triangle}, Chen--Jin--Koh \cite{CJK1997}, and
Brandt--Thomass\'e \cite{brandt2011dense} determined the complete profile:

\begin{theorem}[\cite{AES74,brandt2011dense,haggkvist1982odd,jin1995triangle}]\label{thm:triangle}
\begin{align*}
    \delta_\chi(K_3,2)=\frac{2}{5}, \quad \delta_\chi(K_3,3)=\frac{10}{29}, \quad  \delta_\chi(K_3,r)=\frac{1}{3} \text{ for every $r\geq 4$.}
\end{align*}
\end{theorem}

Nevertheless, beyond cliques and a few special families, the exact chromatic profile remains
poorly understood.  
The problem of determining the chromatic profile for a fixed graph already appears in the work of Erd\H{o}s and Simonovits~\cite{ES1973}, who remarked that it seemed (in full generality) `too complicated' to study.
Allen, B\"ottcher, Griffiths, Kohayakawa and Morris \cite{ABGKM2013} formulated the following general problem and said that ``despite the progress made in recent years,
we still expect it to be extremely difficult.''

\begin{problem}[\rm {\cite[Problem~45]{ABGKM2013}}]\label{problem}
Determine the chromatic profile for every graph $H$.    
\end{problem}

Our main result resolves the case $r=2$ for every graph $H$ with $\chi(H)=3$.
This is the first nontrivial instance of Problem~\ref{problem}.

\begin{theorem}\label{thm-main0}
We have
$$\{\delta_{\chi}(H,2):\chi(H)=3\}=\left\{\frac{1}{2},\frac{2}{5},\frac{2}{7},\frac{1}{4},\frac{2}{9},\frac{1}{5},\frac{2}{11},\frac{1}{6}\right\}.$$  
\end{theorem}
In particular, although there are infinitely many $3$-chromatic graphs, the possible
bipartiteness thresholds form a finite discrete set of only eight values.  Moreover, our
proof gives a structural characterization of the graphs $H$ attaining each value.

\subsection{Related work}

A graph is triangle-free if and only if every vertex neighborhood is independent. 
Generalizing this, {\L}uczak and Thomass\'e \cite{LuczakThomasse2010} introduced locally bipartite graphs, in which every neighborhood is 2-colorable; more generally, a graph is locally $b$-partite if every neighborhood is $b$-colorable. 
Illingworth \cite{Illingworth-JCTB-2022,Illingworth2022} later studied the chromatic profile of locally $b$-partite graphs.

For cliques, the chromatic profile is well understood:
the classical Andr\'{a}sfai--Erd\H{o}s--S\'{o}s Theorem gives $\delta_{\chi}(K_{r+1},r)=\frac{3r-4}{3r-1}$, while later work of Goddard and Lyle \cite{goddard2011dense} and independently Nikiforov \cite{nikiforov2010chromatic} determined the next two thresholds:
$\delta_\chi(K_{r+1}, r+1)=\frac{19r-28}{19r-9} \text{ and }  \delta_\chi(K_{r+1}, r+2)=\frac{2r-3}{2r-1}.$

The triangle case is the base instance of odd cycles.
For $k \in [4]$ and $n>\binom{k+2}{2}(2k+3)(3k+2)$, H\"{a}ggkvist \cite{haggkvist1982odd} proved that any $C_{2k+1}$-free graph $G$ with $\delta(G) > \frac{2}{2k+3}n$ is bipartite. 
Recently, Yuan and Peng \cite{yuan2024minimum} showed that for $k \geq 5$ and $n \geq 21000k$, the condition $\delta(G) > \frac{n}{6}$ already forces a $C_{2k+1}$-free graph to be bipartite. 
Consequently,
\begin{align}\label{ineq-key-1}
\delta_\chi(C_{2k+1},2) = \max\left\{\frac{2}{2k+3},\frac{1}{6}\right\}.
\end{align}

Up to now, the complete characterization of the chromatic profile of odd cycles is still unknown. 
Thomassen \cite{thomassen2007chromatic} proved that \(\delta_\chi(C_5,r)\le \frac{6}{r}\), and in fact obtained a more general estimate for \(\delta_\chi(C_k,r)\). Combined with a result of Ma \cite{ma2016cycles}, this yields, for every fixed \(k\),
$\Omega((k+1)^{-4(r+1)})=\delta_\chi(C_k,r)=O\!\left(\frac{k}{r}\right)$.
For the family of odd cycles $\mathcal{C}_{2k-1}=\{C_3,\ldots,C_{2k-1}\}$, B\"{o}ttcher et al. \cite{Bottcher2010} established the upper bound $\delta_{\chi}(\mathcal{C}_{2k-1}, 3) \le \frac{1}{(2+\epsilon)k}$ for large $k$.
Yan, Peng, and Yuan \cite{YanPengYuan2024} proved $\delta_\chi(C_{2k+1},r)=\frac{1}{2r+2}$ for all $r\geq 3$ and $k\geq 3r+4$.

Related to the chromatic profile is the \emph{chromatic threshold} of a graph \(H\), defined as
\begin{align*}
\delta_\chi(H):=
& \inf \{d: \exists~C=C(H, d)\text{ such that if $G$ is a graph on $n$ vertices}, \\
& \text{with $\delta(G) \geq d n$ and $H \not \subseteq G$, then $\chi(G) \leq C$}\}.
\end{align*}
This is a coarser analogue of the chromatic profile, asking only for bounded chromatic number rather than the exact minimum degree threshold forcing \(r\)-colorability.
The chromatic threshold is by now much better understood: Allen, B\"ottcher, Griffiths, Kohayakawa, and Morris \cite{ABGKM2013} determined it for every graph, 
and more recently Kim, Liu, Shangguan, Wang, Wu, and Xue~\cite{liu-shangguan-wu-xue} established a sharp stability theorem showing that any \(H\)-free graph with \(\delta(G)\ge (\delta_\chi(H)-o(1))n\) and large chromatic number must be structurally close to an extremal configuration.

\subsection{Reduction of Theorem \ref{thm-main0}}
In this subsection, we reduce Theorem \ref{thm-main0} to two separate cases, according to whether $H$ is color-critical. Recall that a graph is \emph{color-critical} if the deletion of some edge decreases its chromatic number.

We first consider the non-color-critical case. In fact, a more general result of Erd\H{o}s and Simonovits \cite{ES1973} implies that $\delta_\chi(H,2)=\frac{1}{2}$ for every non-color-critical graph $H$ with $\chi(H)=3$.

\begin{proposition}[\cite{ES1973}]\label{prop}
If $H$ has no color-critical edges, then 
$\delta_{\chi}(H,\chi(H)-1) = \frac{\chi(H)-2}{\chi(H)-1}.$
\end{proposition}

We now turn to the color-critical case. 
Let $g_{\mathrm{odd}}(H)$ denote the \emph{odd girth} of $H$, that is, the length of a shortest odd cycle in $H$.
If $H$ is color-critical and $\chi(H)=3$, then the value of $\delta_\chi(H, 2)$ is determined by the odd girth of $H$ together with its containment in a finite family of extremal template graphs. More precisely, we prove the following theorem. For a positive integer $n$, we write $[n]:=\{1,2,\ldots,n\}$.

\begin{theorem}\label{thm-main}
    Let $H$ be a color-critical graph with $\chi(H)=3$ and let $\mathcal{G}_{1},\ldots,\mathcal{G}_{5}$, $\mathcal{G}_{[4]},\mathcal{G}_{[5]}$ be the family of graphs defined in Definition \ref{def-1.7} in Section \ref{sub:construction}. Then 
    \[
    \delta_\chi(H,2) =\left\{ \begin{array}{ll}
                  \frac{2}{5}, & g_{\mathrm{odd}}(H)=3; \\[5pt]
                  \frac{2}{7}, & g_{\mathrm{odd}}(H)=5;\\[5pt]
                  \frac{1}{4}, & 
                            g_{\mathrm{odd}}(H)\geq 7, \hbox{$H \not\hookrightarrow \mathcal{G}_{i}$ for some $i\in[4]$;}\\[5pt]
                  \frac{2}{9},& g_{\mathrm{odd}}(H)= 7, \hbox{$H \hookrightarrow \mathcal{G}_{[4]}$;} \\[5pt]
                  \frac{1}{5},& g_{\mathrm{odd}}(H)\geq  9, \hbox{$H \hookrightarrow \mathcal{G}_{[4]}$, and } H \not\hookrightarrow \mathcal{G}_{5};\\[5pt]
                  \frac{2}{11}, & g_{\mathrm{odd}}(H)= 9, \hbox{$H \hookrightarrow \mathcal{G}_{[5]}$;} \\[5pt]
                  \frac{1}{6}, & g_{\mathrm{odd}}(H)\geq 11, \hbox{$H \hookrightarrow \mathcal{G}_{[5]}$}.
                \end{array}\right.
    \]
\end{theorem}

The determination of $\delta_\chi(H,2)$ for color-critical graphs is illustrated by the flowchart in Figure \ref{fig:compact_flow}.

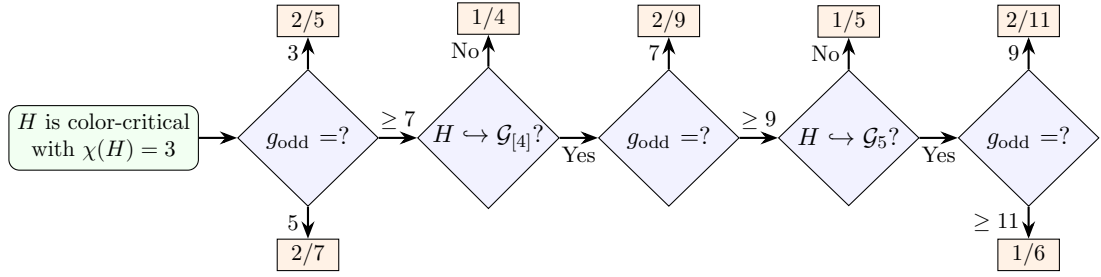
\begin{figure}[H]
\centering
% scale=0.85 确保整体缩放以适应页面宽度
% transform shape 确保字体随之缩放
\begin{tikzpicture}[
    scale=0.85, transform shape, 
    node distance=0.6cm, % 极小的间距
   decision/.style={
        diamond, 
        draw, 
        fill=blue!5, 
        minimum width=2.2cm,   % 【强制】菱形最小宽度
        minimum height=1.8cm,  % 【强制】菱形最小高度
        text width=1.8cm,      % 文字宽度，超过自动换行
        align=center, 
        inner sep=0pt, 
        font=\normalsize     % 字体稍微缩小以适配
        },
    block/.style={rectangle, draw, fill=green!5, align=center, rounded corners, minimum height=2em, font=\small},
    outcome/.style={rectangle, draw, fill=orange!10, minimum width=2.5em, align=center, font=\bfseries\small, inner sep=2pt},
    arrow/.style={-Stealth, thick}
]

    % --- 主轴：所有判断节点排成一行 ---
    
    % 起点
    \node (start) [block] {$H$ is color-critical\\with $\chi(H)=3$};
    
    % Step 1: Girth check
    \node (g1) [decision, right=of start] {$g_{\mathrm{odd}}=$?};
    
    % Step 2: G1-G4 check (强制换行以节省宽度)
    \node (c1) [decision, right=of g1] {$H \hookrightarrow \mathcal{G}_{[4]}$?};
    
    % Step 3: Girth check
    \node (g2) [decision, right=of c1] {$g_{\mathrm{odd}}=$?};
    
    % Step 4: G5 check
    \node (c2) [decision, right=of g2] {$H \hookrightarrow \mathcal{G}_{5}$?};
    
    % Step 5: Final Girth check
    \node (g3) [decision, right=of c2] {$g_{\mathrm{odd}}=$?};

    % --- 结果节点：分布在主轴的上下方 (不占横向空间) ---
    
    % G1 outcomes
    \node (v25) [outcome, above=0.5cm of g1] {$2/5$};
    \node (v27) [outcome, below=0.5cm of g1] {$2/7$};
    
    % C1 outcome
    \node (v14) [outcome, above=0.5cm of c1] {$1/4$};
    
    % G2 outcome
    \node (v29) [outcome, above=0.5cm of g2] {$2/9$};
    
    % C2 outcome
    \node (v15) [outcome, above=0.5cm of c2] {$1/5$};
    
    % G3 outcomes
    \node (v211) [outcome, above=0.5cm of g3] {$2/11$};
    \node (v16) [outcome, below=0.5cm of g3] {$1/6$};

    % --- 连线逻辑 ---

    % 主轴连线 (Yes / >= 路径)
    \draw [arrow] (start) -- (g1);
    \draw [arrow] (g1) -- node[above, font=\small] {$\ge 7$} (c1);
    \draw [arrow] (c1) -- node[below, font=\small] {Yes} (g2);
    \draw [arrow] (g2) -- node[above, font=\small] {$\ge 9$} (c2);
    \draw [arrow] (c2) -- node[below, font=\small] {Yes} (g3);

    % 分支连线 (Outcomes)
    % G1
    \draw [arrow] (g1) -- node[left, font=\small] {3} (v25);
    \draw [arrow] (g1) -- node[left, font=\small] {5} (v27);
    
    % C1 (No分支)
    \draw [arrow] (c1) -- node[left, font=\small] {No} (v14);
    
    % G2
    \draw [arrow] (g2) -- node[left, font=\small] {7} (v29);
    
    % C2 (No分支)
    \draw [arrow] (c2) -- node[left, font=\small] {No} (v15);
    
    % G3
    \draw [arrow] (g3) -- node[left, font=\small] {9} (v211);
    \draw [arrow] (g3) -- node[left, font=\small] {$\ge 11$} (v16);

\end{tikzpicture}
\caption{Structural classification of $\delta_\chi(H,2)$ for color-critical tripartite graphs.}
\label{fig:compact_flow}
\end{figure}

Let $C_{2k+1}[n_1, n_2, \ldots, n_{2k+1}]$ denote the \textit{blow-up} of the odd cycle $C_{2k+1}$ with vertex classes of sizes $n_1, \ldots, n_{2k+1}$, where every two consecutive classes are joined by all possible edges, with indices taken modulo $2k+1$. 

\begin{table}[H]
\centering
\renewcommand{\arraystretch}{1.2}
\begin{tabular}{cccc}
\toprule
Example graph $H$ & Girth & Structural reason & $\delta_\chi(H,2)$ \\
\midrule
$K_3$ & 3 & / & $2/5$ \\
$C_5[1,1,2,2,2]$ & 5 & / & $2/7$ \\
$C_7[1,1,2,2,2,2,2]$ & 7 & $ \not\hookrightarrow \mathcal{G}_{3}$ & $1/4$ \\
$C_7[1,1,1,1,1,1,2]$ & 7 & $ \hookrightarrow \mathcal{G}_{[4]}$ & $2/9$ \\
$C_9[1,1,1,1,1,2,2,2,2]$ & 9 & $\hookrightarrow\mathcal{G}_{[4]}$ and  $ \not\hookrightarrow \mathcal{G}_{5}$ & $1/5$ \\
$C_9[1,1,1,1,1,1,1,1,2]$  & 9 & $\hookrightarrow \mathcal{G}_{5}$ & $2/11$ \\
$C_{11}$ & 11 & / &  $1/6$ \\
\bottomrule
\end{tabular}
\caption{Examples corresponding to the possible values of $\delta_\chi(H,2)$ for color-critical graphs $H$ with $\chi(H)=3$.}
\label{tab:examples}
\end{table}

\subsection{Our constructions}\label{sub:construction}

We define six graph constructions, $G_1(n)$ through $G_6(n)$, as illustrated in Figure \ref{fig-G1-6}.

\begin{figure}[th]
  \centering   
	\begin{minipage}[b]{.3\linewidth} 
	   \centering
	   \includegraphics[scale=0.9]{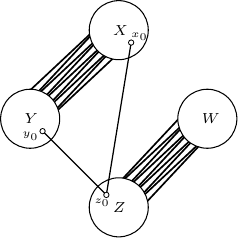}\\[5pt]
       \text{(a) $G_1(n)$}
	\end{minipage}
    	\begin{minipage}[b]{.3\linewidth} 
	   \centering
       \includegraphics[scale=0.9]{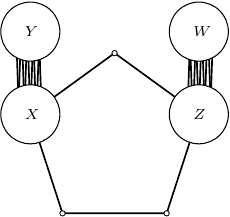}\\[5pt]
       \text{(b) $G_2(n)$}
	\end{minipage}
	\begin{minipage}[b]{.3\linewidth} 
	   \centering
	   \includegraphics[scale=0.9]{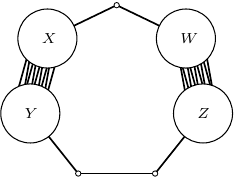}\\[5pt]
       \text{(c) $G_3(n)$}
    \end{minipage}\\[5pt]
    \begin{minipage}[b]{.3\linewidth} 
	   \centering
	   \includegraphics[scale=0.9]{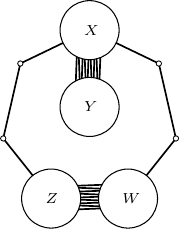}\\[5pt]
       \text{(d) $G_4(n)$}
	\end{minipage}
    	\begin{minipage}[b]{.3\linewidth} 
	   \centering
       \includegraphics[scale=0.9]{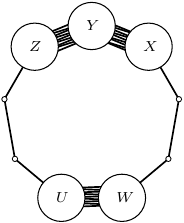}\\[5pt]
       \text{(e) $G_5(n)$}
	\end{minipage}
	\begin{minipage}[b]{.3\linewidth} 
	   \centering
	   \includegraphics[scale=0.55]{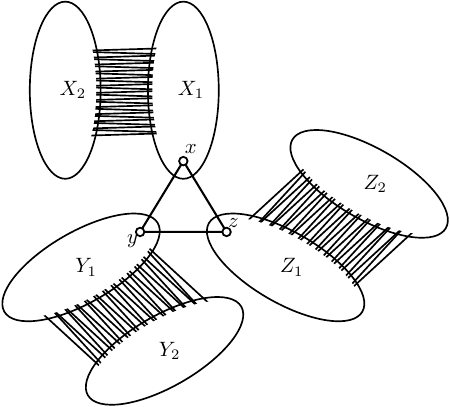}\\[5pt]
       \text{(f) $G_6(n)$}
    \end{minipage}\\[5pt]
  \caption{The extremal constructions $G_1(n)$, $G_2(n)$, $G_3(n)$, $G_4(n)$, $G_5(n)$ and $G_6(n)$.} \label{fig-G1-6}
\end{figure}

\begin{definition}\label{def-1.7}
The graphs $G_1(n), \dots, G_6(n)$ on $n$ vertices are defined as follows:
\begin{itemize}
\item[\rm (i)] $V(G_1(n))=X\cup Y\cup Z\cup W$ with $\lfloor\frac{n}{4}\rfloor=|X|\le|Y|\le|Z|\le|W|=\lceil\frac{n}{4}\rceil$ and 
\[
E(G_1(n))=\{xy\colon x\in X, y\in Y\}\cup \{zw\colon z\in Z, w\in W\}\cup \{z_0x_0,z_0y_0\},
\]
where $x_0\in X,y_0\in Y$ and $z_0\in Z$ are fixed vertices.
\item[\rm (ii)] $V(G_2(n))=X\cup Y\cup Z\cup W\cup \{u_1,u_2,u_3\}$ with $\lfloor\frac{n-3}{4}\rfloor=|X|\le|Y|\le|Z|\le|W|=\lceil\frac{n-3}{4}\rceil$,
\begin{align*}
E(G_2(n)) =&\{xy\colon x\in X, y\in Y\}\cup \{zw\colon z\in Z, w\in W\}\cup\{u_1x\colon x\in X\}\\[3pt]
&\cup\{u_1z\colon z\in Z\}\cup \{u_2x\colon x\in X\}\cup \{u_3z\colon z\in Z\} \cup \{u_2u_3\}. 
\end{align*}
\item[\rm (iii)] $G_3(n)=C_7[n_1,n_2,1,1,n_5,n_6,1]$ with $\lfloor\frac{n-3}{4}\rfloor=n_1\le n_2\le n_5\le n_6=\lceil\frac{n-3}{4}\rceil$.
\item[\rm (iv)] $V(G_4(n))=X\cup Y\cup Z\cup W\cup \{u_1,u_2,u_3,u_4\}$ with $\lfloor\frac{n-4}{4}\rfloor=|X|\le|Y|\le|Z|\le|W|=\lceil\frac{n-4}{4}\rceil$ and 
\begin{align*}
E(G_4(n)) =&\{xy\colon x\in X, y\in Y\}\cup \{zw\colon z\in Z, w\in W\}\cup\{u_1x\colon x\in X\}\\[3pt]
&\cup\{u_2z\colon z\in Z\}\cup \{u_3x\colon x\in X\}\cup \{u_4w\colon w\in W\} \cup \{u_1u_2,u_3u_4\}. 
\end{align*}
\item[\rm (v)] $G_5(n)=C_9[n_1,n_2,n_3,1,1,n_6,n_7,1,1]$ with $\lfloor\frac{n-4}{5}\rfloor=n_1\le n_2 \le n_3\le n_6\le n_7=\lceil\frac{n-4}{5}\rceil$.
\item[\rm (vi)] $V(G_6(n))=X_1\cup X_2\cup Y_1\cup Y_2\cup Z_1\cup Z_2$ with $\lfloor\frac{n}{6}\rfloor=|X_1|\le|X_2|\le|Y_1|\le|Y_2|\le |Z_1|\le|Z_2|=\lceil\frac{n}{6}\rceil$, and 
\begin{align*}
E(G_6(n)) =&\{x_1x_2\colon x_1\in X_1, x_2\in X_2\}\cup \{y_1y_2\colon y_1\in Y_1, y_2\in Y_2\}\cup \{z_1z_2\colon z_1\in Z_1, z_2\in Z_2\} \\[3pt]
& \cup \{xy,yz,xz\},
\end{align*}
where $x\in X_1$, $y\in Y_1$ and $z\in Z_1$ are fixed vertices.
\end{itemize}
Let $\mathcal{G}_i =\bigcup_{n\in\mathbb{N}}G_i(n)$. 
Given a graph $H$ and a family of graphs $\mathcal{G}$, we say $H$ \textit{embeds into} $\mathcal{G}$, denoted by $H \hookrightarrow \mathcal{G}$, if $H$ is a  subgraph of some graph in $\mathcal{G}$. Moreover, write $H \hookrightarrow \mathcal{G}_{[t]}$ if $H \hookrightarrow \mathcal{G}_i$ for every $1\le i\le t$.
\end{definition}

Note that construction $G_6(n)$ was first introduced by H\"{a}ggkvist in \cite{haggkvist1982odd}.
Moreover,
\begin{align}\label{lower-bdd}
    \lim_{n\to\infty}\frac{\delta(G_i(n))}{n}=\frac14 \ \text{for } i\in[4],\qquad
\lim_{n\to\infty}\frac{\delta(G_5(n))}{n}=\frac15,\qquad
\lim_{n\to\infty}\frac{\delta(G_6(n))}{n}=\frac16.
\end{align}

% \begin{remark}
%     Let $H$ be a color-critical graph with $\chi(H)=3$, and let $H^*$
% be the component of $H$ containing a color-critical edge. Then $H^*$ is the
% unique non-bipartite component of H. Moreover, for each $i\in[6]$,
% \[
% H\hookrightarrow \mathcal G_i
% \quad\Longleftrightarrow\quad
% H^*\hookrightarrow \mathcal G_i.
% \]
% Indeed, the implication from left to right is immediate. Conversely, if
% $H^*\subseteq G_i(n_0)$ for some $n_0$, then by taking n sufficiently large, the same copy of $H^*$ can be embedded in $G_i(n)$. 
% Since each $G_i(n)$ has linear minimum degree, all remaining bipartite components of $H-H^*$ can be embedded disjointly in the remaining vertices.
% \end{remark}

\subsection{Reduction of the proof of Theorem \ref{thm-main}}

% In this subsection, we outline the proof of Theorem \ref{thm-main}. The argument splits into two cases according to the odd girth of $H$. Recall that the \emph{odd girth} of $H$, denoted by $g_{\mathrm{odd}}(H)$, is the length of a shortest odd cycle in $H$.
% The case $g_{\mathrm{odd}}(H)=\chi(H)=3$ is handled by the following result, which strengthens Theorem \ref{thm:triangle}.

In this subsection, we outline the proof of Theorem \ref{thm-main}. The proof splits into two scenarios: when $g_{\mathrm{odd}}(H)=\chi(H)=3$, and when $g_{\mathrm{odd}}(H)\geq 5$. The first case is resolved by Theorem \ref{thm-main1}, which in fact establishes a stronger result. To address the second case, we first prove Theorem \ref{theorem-main-0}, which reduces the problem to determining the values of the function $\delta_{\mathrm{ext}}(H,2)$ (defined below). These values are then fully characterized by Theorem \ref{thm-main2}.

We extend Theorem \ref{thm:triangle} to color-critical tripartite graphs with odd girth 3.

\begin{theorem}\label{thm-main1}
Let $H$ be a color-critical graph with $g_{\mathrm{odd}}(H)=\chi(H)=3$. Then 
\begin{align*}
    \delta_\chi(H,2)=\frac{2}{5}, \quad \delta_\chi(H,3)=\frac{10}{29}, \quad  \delta_\chi(H,r)=\frac{1}{3} \text{ for every $r\geq 4$.}
\end{align*}
\end{theorem}

% To prove our main theorem, we utilize the notion of \emph{vertex-extendability}, introduced by Liu, Mubayi, and Reiher \cite{liu2023unified}. 
% We introduce a new concept of the {\it vertex-extendable threshold} of $H$, denoted by $\delta_{\mathrm{ext}}(H,r)$, as the infimum of $c\in (0,1)$ so that for every $H$-free graph $G$ on $n$ vertices, the existence of a vertex $v \in V(G)$ with $\chi(G - v) \leq r$ combined with $\delta(G)\ge cn$ implies that $G$ is $r$-colorable. Precisely,
% \begin{align*}
%     \delta_{\mathrm{ext}}(H, r) := \inf \{c : \delta(G) \geq cn, H \not\subseteq G,~ \exists~v \in V(G) \text{ s.t. } \chi(G-v) \leq r \implies \chi(G) \leq r\}.
% \end{align*}

To prove our main theorem, we use the notion of \emph{vertex-extendability}, introduced by Liu, Mubayi, and Reiher \cite{liu2023unified}. We next define the \emph{vertex-extendable threshold}.

\begin{definition}[vertex-extendable threshold]
The \emph{vertex-extendable threshold} of $H$ with respect to $r$, denoted by $\delta_{\mathrm{ext}}(H,r)$, is the infimum of $c\in (0,1)$ so that for every $H$-free graph $G$ on $n$ vertices, the existence of a vertex $v \in V(G)$ with $\chi(G - v) \leq r$ combined with $\delta(G)\ge cn$ implies that $G$ is $r$-colorable. Precisely,
\begin{align*}
    \delta_{\mathrm{ext}}(H, r) := \inf \{c : \delta(G) \geq cn, H \not\subseteq G,~ \exists~v \in V(G) \text{ s.t. } \chi(G-v) \leq r \implies \chi(G) \leq r\}.
\end{align*}
\end{definition}

The key reduction shows that
$\delta_{\chi}(H,2)$ can be determined by $\delta_\chi(C_{2k+1},2)$ and the vertex-extendable threshold of $H$ if the odd girth of $H$ is at least 5.

\begin{theorem}\label{theorem-main-0}
Let $H$ be a color-critical graph with $\chi(H)=3$. If $g_{\mathrm{odd}}(H)=2k+1$ for $k\geq 2$, then 
\begin{align}\label{ineq-key-2}
\delta_{\chi}(H,2) = \max\left\{\delta_\chi(C_{2k+1},2),\delta_{\mathrm{ext}}(H,2)\right\}. 
\end{align}
\end{theorem}

% For the case where $g_{\mathrm{odd}}(H)=3$, we establish the following stronger structural results.

% \begin{theorem}\label{thm-main1}
% Let $H$ be a color-critical graph with $\chi(H)=3$, and let $h=|H|$. Then 
% \begin{itemize}
%     \item[\rm (i)] every $H$-free graph with $\delta(G)>\frac{2}{5}n$ and $n\geq 15 \times h4^{h}$ must be 2-colorable;
%     \item[\rm (ii)] every $H$-free graph with $\delta(G)>\frac{10}{29}n$ and $n\geq 87 \times h4^{h}$ must be 3-colorable;
%     \item[\rm (iii)] every $H$-free graph with $\delta(G)>\frac{1}{3}n+\varepsilon n$ and $n\geq h 4^h/\varepsilon$ must be 4-colorable.
% \end{itemize}
% \end{theorem}

In view of \eqref{ineq-key-2}, determining $\delta_{\chi}(H,2)$ reduces to establishing tight upper bounds for $\delta_{\mathrm{ext}}(H,2)$.

\begin{theorem}\label{thm-main2}
  Let $H$ be a color-critical graph with $\chi(H)=3$ and $g_{\mathrm{odd}}(H)\geq 5$.  
\begin{itemize}
    \item[\rm (i)] If $H \not\hookrightarrow \mathcal{G}_{i}$ for some $i\in[4]$, then $\delta_{\mathrm{ext}}(H,2)= \frac{1}{4}$;
    \item[\rm (ii)] If $H \hookrightarrow \mathcal{G}_{[4]}$ and $H \not\hookrightarrow \mathcal{G}_5$, then $\delta_{\mathrm{ext}}(H,2)= \frac{1}{5}$;
    \item[\rm (iii)] If  $H \hookrightarrow \mathcal{G}_{[5]}$, then $\delta_{\mathrm{ext}}(H,2)= \frac{1}{6}$.
\end{itemize}
\end{theorem}

With Theorems \ref{thm-main1}, \ref{theorem-main-0}, and \ref{thm-main2} as tools, we show that $\delta_\chi(H,2)$ is governed by the odd girth of $H$ and its containment within a specific family of graphs $\mathcal{G}_1, \dots, \mathcal{G}_6$, as stated in Theorem \ref{thm-main}.

%Let us mention that in Theorem \ref{thm-main},  $\delta_\chi(H,2)$ is determined for all 3-chromatic graph $H$ with a color critical edge. Indeed, it is obvious that all $\delta_\chi(H,2)$ with $t(H)\geq 7$ is determined. Note that if $H\not\subseteq C_7[h,h,1,1,1,1,1]$ and $H\not\subseteq C_7[h,1,1,h,1,1,1]$ and $t(H)\leq 9$, then $\delta_\chi(H,2)=\frac{1}{4}$. Thus we are left with the case $t(H)\leq 9$, $H\subseteq C_7[h,h,1,1,1,1,1]$ or $H\subseteq C_7[h,1,1,h,1,1,1]$. Then by $t(H)\leq 9$, $H\subseteq C_9[h,h,h,h,1,1,1,1,1]$ or $H\subseteq C_9[h,h,h,1,1,h,1,1,1]$. If $H\not\subseteq C_9[h,h,h,1,1,h,h,1,1]$, then $\delta_\chi(H,2)=\frac{1}{5}$. Thus we are left with the case $t(H)\leq 9$, $H\subseteq C_9[h,h,h,1,1,h,1,1,1]$.

\subsection{Notations and organization of the paper}
The \emph{$t$-blow-up} of a graph $G$, denoted by $G[t]$, is obtained by replacing each vertex $v \in V(G)$ with an independent set $I_v$ of size $t$, and connecting two sets $I_u$ and $I_v$ with a complete bipartite graph whenever $uv \in E(G)$.

For a graph $H$, we say that $H$ is {\it $C_{2k+1}$-colorable} if there exist integers $n_1, \ldots, n_{2k+1}$ such that $H \subseteq C_{2k+1}[n_1, n_2, \ldots, n_{2k+1}]$. 
Equivalently, there exists a homomorphism from $H$ to $C_{2k+1}$.
Observe that if $H$ is  $C_{2k+1}$-colorable then $H$ is also $C_{2k-1}$-colorable. 
Let $G$ be a graph with vertex set $V(G)$ and $S\subseteq V(G)$.
We use $G[S]$ to denote the subgraph of $G$ induced by $S$.
For convenience, we shall ignore floor and ceiling signs throughout the proof.

The rest of the paper is organized as follows.
In Section \ref{Section:2}, we present the proofs of Theorems \ref{thm-main1} and \ref{theorem-main-0}.
In Section 
\ref{Section:3}, we prove Theorem \ref{thm-main} using Theorem \ref{thm-main2}. 
The proof of Theorem \ref{thm-main2} is divided into two parts: statements (i) and (ii) are proved in Section \ref{Section:4}, and statement (iii) is established in Section \ref{Section:5}.
In Section~\ref{sec:regular}, we apply our main result to regular Tur\'an numbers.
In Section~\ref{sec:concluding}, we present some open problems related to the chromatic profile of graphs and digraphs.

\section{Proofs of Theorems \ref{thm-main1} and \ref{theorem-main-0}}\label{Section:2}

\subsection{Proof of Theorem \ref{thm-main1}}

For positive integers $s$, $t$, $m$ and $n$, the {\it Zarankiewicz number} $z(m,n;s,t)$ is defined to be the maximum number of edges in a bipartite graph with partite sets  of sizes $m$ and $n$ that contains no complete bipartite subgraph with  $s$ vertices in the part of size $m$ and $t$ vertices in
 the part of size $n$.

\begin{theorem}[\cite{kovari1954problem}]\label{thm:KST}
\[
z(m,n;s,t)<(t-1)^{1/s}mn^{1-1/s}+(s-1)n.
\]
\end{theorem}

\begin{observation}\label{obs}
    Let $H$ be a color-critical graph with $\chi(H)=3$, and let
$H^*$ be the component containing a color-critical edge. Then $H^*$ is the
unique non-bipartite component of $H$. Moreover, for every fixed $\alpha>0$,
any copy of $H^*$ in an $n$-vertex graph $G$ with $\delta(G)\ge \alpha n$ extends
to a copy of $H$, provided $n$ is sufficiently large.
Consequently, for every $r\ge 2$,
\[
\delta_\chi(H,r)=\delta_\chi(H^*,r)
\quad\text{and}\quad
\delta_{\rm ext}(H,r)=\delta_{\rm ext}(H^*,r).
\]
\end{observation}

\begin{proof}
    Let $e$ be a color-critical edge of $H$, and let $H^*$ be the component
containing $e$. Since $\chi(H-e)=2$, the graph $H-e$ is bipartite. Thus $H^*$ is the unique non-bipartite component of $H$.

Let $B:=H-H^*$ and $h=|H|$. Then $B$ is bipartite. Suppose that $G$ is an
$n$-vertex graph with $\delta(G)\ge \alpha n$ and that $G$ contains a copy
$S$ of $H^*$. For $n$ sufficiently large, the graph $G-S$ has minimum degree at least $\alpha n/2$. Hence, by Theorem \ref{thm:KST}, $G-S$ contains a copy of $K_{h,h}$. Since $B\subseteq K_{h,h}$, the copy of $H^*$ extends to a copy of $H$.

Consequently, in graphs with linear minimum degree, being $H$-free is
asymptotically equivalent to being $H^*$-free. Hence, for every $r\ge 2$,
\[
\delta_\chi(H,r)=\delta_\chi(H^*,r)
\quad\text{and}\quad
\delta_{\rm ext}(H,r)=\delta_{\rm ext}(H^*,r).
\]
\end{proof}

In view of Observation \ref{obs}, all assertions below concerning a
color-critical graph $H$ with $\chi(H)=3$ may be proved after replacing $H$ by its unique non-bipartite component $H^*$. Thus, throughout the rest of the paper, unless otherwise stated, we assume without loss of generality that $H$ is connected.

\begin{lemma}\label{lmm-main1}
Let $H$ be a color-critical graph with $\chi(H)=3$, and let $h=|H|$. 
Let $G$ be a graph on $n > h6^h$ vertices with $\delta(G)>n/3+6h$.
If $G$ contains a triangle, then $H\subseteq G$.
\end{lemma}

\begin{proof}
Assume that $G$ contains a triangle $xyzx$ and $\delta(G)=n/3+m$, where $m>6h$.
Then
\[
n\geq |N(x)\cup N(y)\cup N(z)| > 3\times (n/3+m)- |N(x)\cap N(y)|- |N(x)\cap N(z)|- |N(y)\cap N(z)|.
\]
It follows that one of $|N(x)\cap N(y)|$, $|N(x)\cap N(z)|$ and $|N(y)\cap N(z)|$ is greater than $m$. 
Without loss of generality, assume $|N(x)\cap N(y)|> m$.

Let $X\subseteq N(x)\cap N(y)$ with $|X|=m$ and let $Y=V(G)\setminus (X\cup \{x,y\})$. 
Since $d_{Y}(w)\geq n/3+m-(m+1)=n/3-1$ for each $w\in X$, we infer that 
\begin{align*}
e(G[X,Y])\geq m (n/3-1)>(h-1)^{1/h}m(n-m)^{1-1/h}+(h-1)(n-m),
\end{align*}
where the last inequality holds if $n>m>6h$ and $n>h6^h$.
Then by Theorem \ref{thm:KST}, $G[X,Y]$ contains a copy $K$ of $K_{h,h}$. Since $H$ is 3-chromatic and contains a color-critical edge, one can find a copy of $H$ in $G[V(K)\cup \{x,y\}]$. 
\end{proof}

\begin{proof}[\textbf{Proof of Theorem \ref{thm-main1}}]
Since $g_{\mathrm{odd}}(H)=3$, we have $K_3\subseteq H$, which implies $\delta_\chi(H,r)\ge \delta_\chi(K_3,r)$ for every $r\ge 2$.

For the upper bound, let $G$ be an $H$-free graph on $n$ vertices with $\delta(G)=n/3+m$, where $m>6h$.
Then by Lemma \ref{lmm-main1}, $G$ is $K_3$-free.
Thus $\delta(G)>\delta_\chi(K_3,r)$ implies $G$ is $r$-colorable.
This proves $\delta_\chi(H,r) \le \delta_\chi(K_3,r)$, and the result follows.
\end{proof}

\subsection{Proof of Theorem \ref{theorem-main-0}}

We first need an easy but important fact.
\begin{fact}\label{fact-key}
    Let $H$ be a connected  color-critical graph with $\chi(H)=3$. If $g_{\mathrm{odd}}(H)=2k+1$ with $k\geq 1$ and $h=|H|$, then $H\subseteq C_{2k+1}[1,1,h,\ldots,h]$. 
    Moreover, $H\subseteq C_{2i+1}[1,1,h,\ldots,h]$ for every $2\le i\le k$.
\end{fact}

\begin{comment}

\begin{proof}
Let $xy$ be a color-critical edge of $H$ and let $H'=H-xy$. 
Define 
\[
V_i := \{v\in V(H')\colon d_{H'}(v,x)=i\},
\]
where $d_{H'}(v,x)$ denotes the distance between $v$ and $x$ in $H'$. Since $g(H)=2k+1$, every cycle of length $2k+1$ goes through $xy$. It implies  that $y\in V_{2k}$. 

Since $H'$ is bipartite, each $V_i$ is an independent set.  Now define $\phi(x)=0$ and 
\[
\phi(V_i)=i \mbox{ for } i=1,2,\ldots,2k-1
\]
and for $i\geq 2k$ define $\phi(V_i)=2k-1$ for odd $i$ and $\phi(V_i\setminus \{y\})=2k-2$ for even $i$. Clearly $\phi$ is a homomorphism from $H$ to $C_{2k+1}$ with $\phi^{-1}(0)=\{x\}$ and $\phi^{-1}(2k)=\{y\}$. Thus $H\subseteq C_{2k+1}[1,1,h,h,\ldots,h]$. 
\end{proof}
\end{comment}

\begin{proof}
Let $xy$ be a color-critical edge of $H$ and let $H'=H-xy$. 
Define 
\[
V_i := \{v\in V(H')\colon d_{H'}(v,x)=i\},
\]
where $d_{H'}(v,x)$ denotes the distance between $v$ and $x$ in $H'$. Since $g_{\mathrm{odd}}(H)=2k+1$, every cycle of length $2k+1$ passes through $xy$. This implies $y\in V_{2k}$. 
Since $H'$ is bipartite, each $V_i$ is an independent set.

  Let $V(C_{2k+1}) = \{0, 1, \dots, 2k\}$ with edges $(j, j+1)$ modulo $2k+1$. 
  We define a mapping $\phi: V(H) \to V(C_{2k+1})$ as follows:
  For $0 \le i \le 2k-1$, let $\phi(v) = i$ for all $v \in V_i$. 
  Note that $\phi(x) = 0$.
  For the layer $V_{2k}$, we distinguish $y$ from other vertices:$$\phi(y) = 2k, \quad \text{and} \quad \phi(v) = 2k-2 \text{ for all } v \in V_{2k} \setminus \{y\}.$$
  For $i \ge 2k+1$, we map the vertices to $2k-1$ and $2k-2$ alternately to preserve the bipartite structure of the tail:$$\phi(v) = \begin{cases} 
2k-1 & \text{if } i \text{ is odd}, \\
2k-2 & \text{if } i \text{ is even},
\end{cases} \quad \text{for } v \in V_i.$$
  It is easy to verify that $\phi$ is a homomorphism from $H$ to $C_{2k+1}$.
  Therefore, $H$ is a subgraph of the blow-up $C_{2k+1}$ where the vertices $0$ and $2k$ have capacity $1$, and all other vertices have capacity at most $h$. 
  This confirms $H \subseteq C_{2k+1}[1,1,h,h,\ldots,h]$.
  Moreover, for every $2\le i\le k$, the same layering argument gives a homomorphism from $H$ to $C_{2i+1}$.
  Hence, $H\subseteq C_{2i+1}[1,1,h,\ldots,h]$.
\end{proof}

\begin{remark}
   If $H$ is not color-critical, then the condition $g_{\mathrm{odd}}(H)=2k+1$ is not equivalent to $H$ being a subgraph of a blow-up of $C_{2k+1}$. For instance, the Petersen graph has odd girth $5$, but it is not a subgraph of any blow-up of $C_5$.
\end{remark}

% We also need the minimum degree version of the celebrated graph removal lemma.

% \begin{lemma}[\cite{gishboliner2024minimum}]\label{lem-degreerem}
% Let $G$ be a graph with $\delta(G)>\frac{n}{2k+1}$. 
% For every $\varepsilon > 0$, there exists $\delta > 0$ (depending polynomially on $\varepsilon$) such that if $G$ contains at most $\delta n^{2k+1}$ copies of $C_{2k+1}$, then $G$ can be made $C_{2k+1}$-free by deleting at most $\varepsilon n^2$ edges.
% \end{lemma}

\begin{theorem}[Removal Lemma \cite{ruzsa1978triple}]\label{thm:removal}
Let $H$ be a graph on $h$ vertices. For every $\varepsilon>0$, there exists a $\delta>0$ such that the following holds. If $G$ is an $n$-vertex graph with fewer than $\delta n^h$ copies of $H$, then one can remove at most $\varepsilon n^2$ edges from $G$ to make it $H$-free.
\end{theorem}

For two graphs $F$ and $H$, we define  ${\rm ex}(n, F, H)$
 as the maximum possible number of copies of $F$ in an $H$-free graph on $n$ vertices. Alon and Shikhelman~\cite{alon2016many} proved the following result.

\begin{theorem}[\cite{alon2016many}]\label{thm:AS16}
Let $F$ be a graph on $f$ vertices. If $H$ is a subgraph of a blow-up of $F$, then there exists some $\varepsilon:=\varepsilon(F,H)>0$ such that 
\[
{\rm ex}(n, F, H) \leq n^{f-\varepsilon}.
\] 
\end{theorem}

\begin{lemma}\label{lmm:upper-bound}
    Suppose $H$ is a subgraph of a blow-up of $F$, then 
\begin{align*}
    \delta_\chi(H,r)\le\max\{\delta_\chi(F,r),\delta_{\mathrm{ext}}(H,r)\}.
\end{align*}
\end{lemma}

\begin{proof}
    Let 
\[
c=\max\{\delta_\chi(F,r),\delta_{\mathrm{ext}}(H,r)\}.
\]
Let $\gamma>0$ be sufficiently small, and suppose that $G$ is an $H$-free graph on $n$ vertices with $\delta(G)>(c+2 \gamma) n$.
We will show that $G$ is $r$-colorable for sufficiently large $n$. 

Set $\varepsilon = \gamma^3$. Let $\delta > 0$ be the constant provided by Removal Lemma (Theorem \ref{thm:removal}) for this $\varepsilon$.
Since $H$ is a subgraph of a blow-up of $F$, Theorem \ref{thm:AS16} implies that for sufficiently large $n$, the number of copies of $F$ in $G$ is at most $\delta n^{|F|}$.
By Theorem \ref{thm:removal}, there exists a subset $E_{rem} \subseteq E(G)$ with $|E_{rem}| \leq \varepsilon n^2$ such that $G':=G-E_{rem}$ is $F$-free.

Let
\[
T=\{v \in V(G')\colon d_{G'}(v)\leq (c+\gamma)n \}.
\]
Since $|E_{rem}| \le \gamma^3 n^2$, we have $|T|\gamma n\leq 2|E_{rem}| \leq 2\gamma^3 n^2$, which implies $|T|\leq 2\gamma^2 n$.
Consider the graph $G_0 := G' - T$. Then for any $v \in V(G_0)$,
\begin{align}\label{ineq-4.0}
d_{G_0}(v) \geq (c+\gamma)n - |T| \geq (c+\gamma)n-2\gamma^2n \geq \left(c+\gamma/2\right)n.
\end{align}
Since $G_0$ is $F$-free, by $\delta(G_0)> (\delta_{\chi}(F,r)+\gamma/2)n$ we infer that $G_0$ is $r$-colorable. 

Next, we show that $G_1 := G-T$ is $r$-colorable. Let $V(G_1)=\{u_1,\ldots,u_m\}$ and let $E_{i}$ be the set of edges in $E_{rem}$ incident to $u_i$ within $G_1$.
Define a sequence of graphs $J_0 = G_0$ and $J_i = J_{i-1} \cup E_i$ for $i \in [m]$. Note $J_m = G_1$.
We proceed by induction. The base case $J_0$ is $r$-colorable.
Suppose $J_{i-1}$ is $r$-colorable. Since  $J_i - u_i$ is a subgraph of $J_{i-1}$, we have $\chi(J_i - u_i) \le \chi(J_{i-1}) \le r$. 
Since $J_i$ is $H$-free and satisfies $\delta(J_i) \ge (\delta_{\mathrm{ext}}(H, r)+\gamma/2)n$, 
by the definition of $\delta_{\mathrm{ext}}(H, r)$, the existence of such a vertex $u_i$ implies that $J_i$ is $r$-colorable. 
Thus, by induction $G_1$ is $r$-colorable.

Finally, enumerate $T$ as $\{v_1,\ldots,v_t\}$. Define $L_0 = G_1$ and $L_j = G[V(G_1) \cup \{v_1, \dots, v_j\}]$.
Assume $L_{j-1}$ is $r$-colorable. Since $L_j - v_j = L_{j-1}$, we have $\chi(L_j - v_j) \le r$.
Also, $L_j \subseteq G$ is $H$-free and $\delta(L_j) > (\delta_{\mathrm{ext}}(H, r)+\gamma/2)|L_j|$.
Thus, $L_j$ is $r$-colorable.
Consequently, $G = L_t$ is $r$-colorable.
This proves
\[
\delta_{\chi}(H,r) \leq \max\{\delta_\chi(F,r),\delta_{\mathrm{ext}}(H,r)\}.
\]
\end{proof}

\begin{proof}[\textbf{Proof of Theorem \ref{theorem-main-0}}]
Since $g_{\mathrm{odd}}(H) = 2k+1$, by Fact \ref{fact-key}, $H$ is a subgraph of a blow-up of $C_{2k+1}$. It follows from Lemma \ref{lmm:upper-bound} that
\[
\delta_{\chi}(H,2) \leq \max\{\delta_\chi(C_{2k+1},2),\delta_{\mathrm{ext}}(H,2)\}. 
\]

For the lower bound, since $C_{2k+1} \subseteq H$, any graph that is $C_{2k+1}$-free is necessarily $H$-free. 
Thus, $\delta_{\chi}(H,2) \geq \delta_\chi(C_{2k+1},2)$.
Moreover,  by the definition of $\delta_{\mathrm{ext}}(H,2)$, for any sufficiently large $n$ there is an $H$-free $n$-vertex graph $G_0$ with minimum degree at least $(\delta_{\mathrm{ext}}(H,2)-o(1))n$ that is non-bipartite and $G_0-v$ is bipartite for some $v\in V(G_0)$. Thus $\delta_{\chi}(H,2) \geq \delta_{\mathrm{ext}}(H,2)$. Therefore,
\[
\delta_{\chi}(H,2) \geq \max\{\delta_\chi(C_{2k+1},2),\delta_{\mathrm{ext}}(H,2)\}
\]
and the theorem follows.
\end{proof}

\section{Proof of Theorem  \ref{thm-main} (assuming Theorem \ref{thm-main2})}\label{Section:3}

Let us prove our main theorem by assuming Theorem \ref{thm-main2}.

\begin{proof}[\textbf{Proof of Theorem \ref{thm-main}}]
 Let $H$ be a color-critical graph with $\chi(H)=3$ and let $h=|H|$.  
 If $g_{\mathrm{odd}}(H)=3$, then Theorem \ref{thm-main1} immediately implies that $\delta_{\chi}(H,2)=2/5$. 
 We may therefore assume that $g_{\mathrm{odd}}(H)\geq 5$.

 Suppose that $H \not\hookrightarrow \mathcal{G}_{i}$ for some $i\in[4]$. 
 Then Theorem \ref{thm-main2} (i) gives $\delta_{\mathrm{ext}}(H,2) = 1/4$. 
 
 If $g_{\mathrm{odd}}(H)= 5$, then by \eqref{ineq-key-1} and  \eqref{ineq-key-2},
 \[
\delta_{\chi}(H,2) =\max\left\{\delta_\chi(C_5,2),\delta_{\mathrm{ext}}(H,2)\right\}  = \max\{2/7,1/4\}=2/7.
\]

 If instead $g_{\mathrm{odd}}(H)\geq 7$, then again by \eqref{ineq-key-1} and  \eqref{ineq-key-2},
\[
\delta_{\chi}(H,2) \leq \max\left\{\delta_\chi(C_7,2),\delta_{\mathrm{ext}}(H,2)\right\} = \max\{2/9,1/4\}=1/4.
\]

 On the other hand, since $H \not\hookrightarrow \mathcal{G}_{i}$ for some $i\in[4]$, we have $\delta_{\chi}(H,2) \geq \lim\limits_{n\rightarrow \infty}\frac{\delta(G_i(n))}{n}=\frac{1}{4}$. 
 Hence, $\delta_{\chi}(H,2) =1/4$ whenever $g_{\mathrm{odd}}(H)\geq 7$ and $H \not\hookrightarrow \mathcal{G}_{i}$ for some $i\in[4]$.

  Next assume that $H \hookrightarrow \mathcal{G}_{[4]}$ but $H \not\hookrightarrow \mathcal{G}_{5}$.
  Then Theorem \ref{thm-main2} (ii) yields $\delta_{\mathrm{ext}}(H,2)= 1/5$.

  If $g_{\mathrm{odd}}(H)= 7$, combining \eqref{ineq-key-1} and \eqref{ineq-key-2} we have
\[
\delta_{\chi}(H,2) =\max\left\{\delta_\chi(C_7,2),\delta_{\mathrm{ext}}(H,2)\right\}  = \max\{2/9,1/5\}=2/9.
\]

 If $g_{\mathrm{odd}}(H)\geq 9$,  then 
\[
\delta_{\chi}(H,2) \leq \max\left\{\delta_\chi(C_9,2),\delta_{\mathrm{ext}}(H,2)\right\}  = \max\{2/11,1/5\}=1/5.
\]
 Since $H \not\hookrightarrow \mathcal{G}_{5}$, we have
$\delta_{\chi}(H,2) \geq \lim\limits_{n\rightarrow \infty}\frac{\delta(G_5(n))}{n}=\frac{1}{5}$.  
Therefore $\delta_{\chi}(H,2) =1/5$ whenever $g_{\mathrm{odd}}(H)\geq  9$, $H \hookrightarrow \mathcal{G}_{[4]}$ and $H \not\hookrightarrow \mathcal{G}_{5}$.

Finally, assume that $H \hookrightarrow \mathcal{G}_{[5]}$. Then Theorem \ref{thm-main2} (iii) gives $\delta_{\mathrm{ext}}(H,2)= 1/6$. 

If $g_{\mathrm{odd}}(H)= 9$, then by  \eqref{ineq-key-1} and  \eqref{ineq-key-2} we conclude that
\[
\delta_{\chi}(H,2) =\max\left\{\delta_\chi(C_9,2),\delta_{\mathrm{ext}}(H,2)\right\} = \max\{2/11,1/6\}=2/11.
\]

If $g_{\mathrm{odd}}(H)\geq 11$, then 
\[
\delta_{\chi}(H,2) \leq \max\left\{\delta_\chi(C_{11},2),\delta_{\mathrm{ext}}(H,2)\right\} = \max\{2/13,1/6\}=1/6.
\]
Moreover, $g_{\mathrm{odd}}(H)\geq 11$ implies that $H \not\hookrightarrow \mathcal{G}_{6}$. Thus,  
$\delta_{\chi}(H,2) \geq \lim\limits_{n\rightarrow \infty}\frac{\delta(G_6(n))}{n}=\frac{1}{6}$ and $\delta_{\chi}(H,2) =1/6$ follows.
\end{proof}

\section{Proof of Theorem \ref{thm-main2} (i), (ii)}\label{Section:4}

% \begin{lemma}[Gao, Liu, Wu and Xue, \cite{rainbow}]\label{lem-intersection}
%     Let $t\in \mathbb{N}$ and $\varepsilon\in (0,1)$. Then there exist $\alpha =\alpha(\varepsilon,t)$ and $m=m(\varepsilon,t)$ such that for any $V_1, V_2, \cdots, V_m\subseteq [n]$ each with size at least $\varepsilon n$, there exist $1\le i_1< i_2<\dots < i_t\le m$ with
%     $|\bigcap_{j=1}^{t} V_{i_j}| \ge \alpha n$.
% \end{lemma}

The following lemma first appeared in Gao, Liu, Wu and Xue \cite{Gaoetal-arxivorg2026+}. For the sake of completeness, we include a proof in the Appendix.
% \begin{lemma}[Gao, Liu, Wu and Xue, \cite{rainbow}]\label{lem-intersection}
%     Let $t \in \mathbb{N}$ and $\gamma \in(0,1)$. Define $\alpha_i$ and $m_i$ recursively as $\alpha_i=\alpha_{i-1}^2 / 4$ and $m_i= 2 \alpha_{t-i}^{-1} m_{i-1}$ with initial conditions $\alpha_0=\gamma$ and $m_0=1$. Let $V_1, V_2, \cdots, V_m$ be $m$ subsets of $[n]$ with size at least $\varepsilon n$. For each $t \in \mathbb{N}$, if $m \geq m_t$, then there exists an index set $I \subseteq[m]$ of size $2^t$ such that
% $$
% \Big|\bigcap_{i \in I} V_i\Big| \geq \alpha_t n.
% $$
% \end{lemma}

% \begin{corollary}\label{cor-intersection}
%     Let $t\in \mathbb{N}$ and $\varepsilon\in (0,1)$.
%     Let $G$ be an $n$-vertex graph with $\delta(G)\ge \varepsilon n$.
%     For every subset $X\subseteq V(G)$ with $|X|\ge m(\varepsilon,t)$, there exist $x_1,\ldots,x_t\in X$ such that $|\bigcap_{i\in[t]}N(x_i)|\ge \alpha_t n$.
% \end{corollary}

\begin{lemma}[\cite{Gaoetal-arxivorg2026+}]\label{lem-intersection}
Let $t \in \mathbb{N}$ and $\varepsilon \in (0,1)$. Let $G$ be a bipartite graph with partite sets $A$ and $B$, $|A|+|B|\leq n$. Suppose that $d(x)\geq \varepsilon n$ for each $x\in A$ and $|A|\geq t/\varepsilon$.
Then there exists a copy of $K_{t,\gamma n}$ in $G[A,B]$ with $\gamma = \left(\frac{\varepsilon}{e}\right)^{t+1}$.
\end{lemma}

The following pairing lemma will be used to find a copy of blow-up of $P_3$ from three large vertex sets with suitable neighborhood conditions.

\begin{lemma}[Pairing Lemma]\label{lmm:pairing}
    Let $h\in \mathbb{N}$ and $\varepsilon, \gamma\in (0,1)$. Let $G$ be a graph on $n$ vertices. 
    Let $A$, $B$ and $Z$ be pairwise disjoint sets of vertices satisfying $|A|,|B|, |Z| \ge \gamma n$. 
    Suppose that for every $v\in A\cup B$, we have $|N(v)\cap Z|\ge (1/2+\varepsilon)|Z|$.
    Then there exist subsets $A' \subseteq A$, $B' \subseteq B$, and $Z' \subseteq Z$ with $|A'| = |B'| = |Z'| = h$ such that $(A',Z',B')$ forms a copy of $P_3[h]$.
\end{lemma}

\begin{proof}
Choose distinct vertices $\{a_1,\ldots,a_{\gamma n}\}\in A$ and $\{b_1,\ldots,b_{\gamma n}\}\in B$, and let $M=\left\{(a_i,b_i)\colon 1\le i\le \gamma n\right\}$.
Define an auxiliary bipartite graph $F$ on partite sets $M$ and $Z$ with the edge set $\{\{(a_i,b_i),z\}\colon a_iz,b_iz\in E(G)\}$. 
Since  $|N(a_i)\cap N(b_i)\cap Z|\geq 2\varepsilon |Z|$, we have 
\[
e(F) \geq |M| \cdot 2\varepsilon |Z| =\gamma n \cdot 2\varepsilon |Z|\geq 2\varepsilon \gamma^2 n^2.
\]
By Theorem \ref{thm:KST}, there exists a copy of $K_{h,h}$ in $F$. Hence we obtain a $P_3[h]$ and the lemma holds.    
\end{proof}

\begin{lemma}\label{lem-0}
Let $H$ be a color-critical graph with $\chi(H)=3$ and let $|H|=h$. 
For sufficiently large $n$, if $G$ is an $n$-vertex $H$-free graph with $\delta(G) \geq cn$, then $G$ is $C_3[1,1,2h/c]$-free.
\end{lemma}

\begin{proof}
Suppose for contradiction that there exist $x,y\in V(G)$ and $Z\subseteq V(G)\setminus \{x,y\}$ with $|Z|=2h/c$ such that $(x, y, Z)$ forms a copy of  $C_3[1,1,2h/c]$. Let $Z'= V(G)\setminus (Z\cup \{x,y\})$. Note that for any $z\in Z$, $d(z, Z')\geq cn-2h/c-2>cn/2$. Applying Lemma \ref{lem-intersection} to $G[Z,Z']$, there exists a copy $K_{h,h}$ with parts $S\subseteq Z$ and $T\subseteq Z'$. Now $G[S\cup T\cup \{x,y\}]$ contains $H$ as a subgraph, a contradiction. 
\end{proof}

Our constructions \(G_1(n),\ldots,G_5(n)\), together with \eqref{lower-bdd}, show that \(\delta_{\mathrm{ext}}(H,2)\ge \frac{1}{4}\) whenever \(H \not\hookrightarrow \mathcal{G}_i\) for some \(i\in[4]\), and that \(\delta_{\mathrm{ext}}(H,2)\ge \frac{1}{5}\) whenever \(H \hookrightarrow \mathcal{G}_{[4]}\) but \(H \not\hookrightarrow \mathcal{G}_5\).
It remains to establish the corresponding upper bounds.

\begin{proof}[\textbf{Proof of Theorem \ref{thm-main2} (i)}]
Since $g_{\mathrm{odd}}(H)\geq 5$, Fact \ref{fact-key} implies $H\subseteq C_5[1,1,h,h,h]$.  
Let $G$ be an $H$-free graph with $\delta(G)\geq (1/4+\varepsilon)n$. Suppose that $G-u$ is a bipartite graph on bipartite sets $X$ and $Y$. 
We aim to show that $G$ is bipartite.
Suppose for the sake of contradiction that $G$ is not bipartite. Then $u$ must have neighbors in both $X$ and $Y$.

Without loss of generality, assume  $|N(u)\cap X|\geq |N(u)\cap Y|\geq 1$. 
Fix some $v\in N(u)\cap Y$ and let $A=N(u)\cap X$, $B=N(v)\cap X$. 
Then $|A|\geq n/8$ and $|B|>n/4$.  
Let $h=|H|$. 
By Lemma \ref{lem-0},  $G$ is $C_3[1,1,8h]$-free and thereby  $|A\setminus B|\ge n/9$. 

Since $|N(x)\setminus \{u,v\}|\geq n/4$ for each $x\in A\setminus B$, by Lemma \ref{lem-intersection}  there is a $K_{h, \gamma n}$ with some $\gamma>0$, parts $A'\subseteq A\setminus B$ and $C\subseteq Y\setminus \{v\}$, where $|A'|=h$ and $|C|=\gamma n$. 
Now we distinguish two cases. 

\begin{mycase}{Case 1.}
    $|N(x) \cap B|\geq \varepsilon n$ for each $x\in C$.
\end{mycase}

In this case, $e(G[B,C])\geq \varepsilon \gamma  n^2$. By Theorem \ref{thm:KST}, there exists a copy of $K_{h,h}$ in  $G[B,C]$. 
Together with $A', u$ and $v$, we obtain a $C_5[1,1,h,h,h]$ in $G$, which contains $H$ as a subgraph, a contradiction.  

\begin{mycase}{Case 2.}
    $|N(y_0)\cap B|<\varepsilon n$ for some $y_0\in C$.
\end{mycase}

Let $D=N(y_0)\cap X$. Then $|D|>(1/4+\varepsilon) n-1$ and $|B\cap D|\le |N(y_0)\cap B|<\varepsilon n$.
It follows that $|X|\geq |B|+|D|-|B\cap D|\geq n/2$ and $|Y|<n/2$. Hence, for any $x\in A\cup B$, 
\[
|N(x)\cap Y|\geq n/4+\varepsilon n-1 \geq \left(1/2+\varepsilon\right)|Y|.
\]
Apply the Pairing Lemma (Lemma \ref{lmm:pairing}) to $(A\setminus B,B,Y)$,  we obtain a copy of $P_3[h]$. Together with $u$ and $v$, we obtain a copy of $C_5[1,1,h,h,h]$, which contradicts the fact that $G$ is $H$-free.  Thus $\delta_{\mathrm{ext}}(H,2)\leq 1/4$ follows.
\end{proof}

\begin{lemma}\label{lem-4.4}
Let $c>0$ be a constant and $H$ be a graph on $h$ vertices.
Let $G$ be an $H$-free graph on $n$ vertices for sufficiently large $n$. 
Then the following hold:
\begin{itemize}
    \item[\rm (i)] If $G$ is $G_2(4h+3)$-free, then $G$ is $C_5[2h/c,1,2h/c,1,1]$-free.
    \item[\rm (ii)] If $G$ is $G_4(4h+4)$-free, then $G$ is $C_7[h,h,1,1,2h/c,1,1]$-free. 
\end{itemize}
\end{lemma}

\begin{proof}
Suppose that $(X,u_1,Y,u_2,u_3)$ forms a $C_5[2h/c,1,2h/c,1,1]$ in $G$. Let $Z=V(G)\setminus (X\cup Y\cup \{u_1,u_2,u_3\})$.  Note that for any $x\in X\cup Y$, \[
d(x,Z)\geq cn - 4h/c-3>cn/2.
\]
Since $|X|=|Y| = 2h/c$,  by Lemma \ref{lem-intersection} there exist a copy of $K_{h,\gamma n}$ with parts $S_X\subseteq X$, $T_X\subseteq Z$, and a $K_{h,\gamma n}$ with parts $S_Y\subseteq Y$, $T_Y\subseteq Z$. Choose disjoint sets $T_X'\subseteq T_X$ and $T_Y'\subseteq T_Y$ 
with $|T_X'|=|T_Y'|=h$. Then $G[S_X\cup T_X'\cup S_Y\cup T_Y'\cup \{u_1,u_2,u_3\}]$ forms a copy of $G_2(4h+3)$, a contradiction. Thus (i) holds. By the same argument, one can show that (ii) holds as well. 
\end{proof}

Let  $R$ be a graph on the vertex set $[m]$. Let $V_1, V_2,\ldots,V_m$ be  disjoint vertex sets with $|V_i|=h_i$, $i=1,2,\ldots,m$. 
Define the blow-up graph $R[h_1,h_2,\ldots,h_m]$ as the graph $H$ on the vertex set $V_1\cup V_2\cup \ldots\cup V_m$ such that $H[V_i,V_j]$ induces a complete bipartite graph if $ij\in E(R)$ and $H[V_i,V_j]$ induces an empty graph if $ij\notin E(R)$.

\begin{lemma}\label{lem-5.1} 
Let  $R$ be a graph on the vertex set $[m]$. Let $H=R[h_1,h_2,\ldots,h_m]$ be a blow-up of $R$ and let $G$ be an $H$-free $n$-vertex graph with $\delta(G)\geq cn$ for $n$ sufficiently large. 
Suppose that $m$ has degree one in $R$ and $m-1$ is the only neighbor of $m$ in $R$. 
Let $R'=R-\{m\}$. Then $G$ is $R'[h_1,\ldots,h_{m-2},\tilde{h}_{m-1}]$-free with $\tilde{h}_{m-1}=2h_{m-1}/c$.
\end{lemma}
\begin{proof}
    Suppose for contradiction that $G$ contains $R'[h_1,\ldots,h_{m-2},\tilde{h}_{m-1}]$ as a subgraph. Let $U=V(G)\setminus (V_1 \cup\ldots \cup  V_{m-2}\cup \tilde{V}_{m-1})$. Then for any $x\in \tilde{V}_{m-1}$, we have
    \[
    \deg_{U}(x) \geq cn-h_1-\cdots - h_{m-2}-\tilde{h}_{m-1}\geq cn/2. 
    \]
        Apply Lemma \ref{lem-intersection} to $G[\tilde{V}_{m-1},U]$, one can find a copy of $K_{h_{m-1},h_m}$ with parts $T_{m-1}\subseteq \tilde{V}_{m-1}$ and $T_m\subseteq U$, where $|T_{m-1}|=h_{m-1}$ and $|T_{m}|=h_{m}$. Then $(V_1,\ldots,V_{m-2},T_{m-1},T_m)$ forms a copy of $H$, a contradiction. 
\end{proof}

\begin{remark}
   Lemma \ref{lem-5.1} states that in a graph $G$  with linear minimum degree, to find a subgraph $H$  contained in a blow-up of $R$, it suffices to find a blow-up of the 2-core of $R$, where the 2-core of $R$ is the maximal induced subgraph in which every vertex has degree at least 2. 
   Therefore, in the subsequent proof, we always consider the existence of a blow-up of the 2-core of $R$.
\end{remark}

\begin{proof}[\textbf{Proof of Theorem \ref{thm-main2} (ii)}]
Let $G$ be an $H$-free graph with $\delta(G)\geq (1/5+\varepsilon)n$. Fix some  $\gamma>0$ such that 
$\gamma \leq \left(\frac{\varepsilon}{4e}\right)^{2h/\varepsilon+1}$. Since $H\subseteq G_i(4h+4)$ for $i=1,2,3,4$, by Lemma \ref{lem-4.4} we infer that 
\begin{itemize}
    \item[\rm \textbf{A1}] $G$ is $C_5[10h,1,10h,1,1]$-free;
    \item[\rm \textbf{A2}] $G$ is $C_7[h,h,1,1,10h,1,1]$-free.
\end{itemize}
Note that $H\subseteq G_4(4h+4)$ implies $g_{\mathrm{odd}}(H)\geq 7$. By $H\subseteq  G_1(4h+4)$, we infer that 
\begin{itemize}
    \item[\rm \textbf{A3}] $G$ is $C_7[1,1,1,h,h,h,h]$-free and $C_5[1,1,1,h,h]$-free.
\end{itemize}
 Suppose $G-u$ is bipartite. We shall show that $G$ is also bipartite.

Let $(X,Y)$ be a bipartition of $G-u$. Let $U_X=N(u)\cap X$ and $U_Y=N(u)\cap Y$. Without loss of generality, assume $|U_X|\ge |U_Y|$. Then $|U_X|\ge n/10$. Suppose that $G$ is not bipartite. Then $U_Y\neq \varnothing$. Choose $v\in U_Y$. 
Since $\delta(G)\geq (1/5+\varepsilon )n$, each $x\in U_X$ has at least  $n/5$ neighbors in $Y\setminus \{v\}$. 
  As $|U_X|\geq n/10$, by an averaging argument, there exists a vertex $v'\in Y\setminus \{v\}$ such that 
  \[
  |N(v')\cap U_X| \geq \frac{1}{|Y\setminus \{v\}|}\cdot |U_X| \cdot \frac{n}{5}\geq \frac{1}{n}\cdot \frac{n}{10}\cdot  \frac{n}{5}=\frac{n}{50}\geq \gamma n.
  \]
Let $T\subseteq N(v')\cap U_X$ with $|T|=\gamma n$ and let  $A=(N(v)\cap X)\setminus T$ and $B=(N(v')\cap X)\setminus (A\cup T)$. 
Clearly $|A|\geq n/5$. If $|N(v')\cap A|\geq \gamma n$, then $(T,v',N(v')\cap A, v,u)$ would form a copy of $C_5[\gamma n,1,\gamma n,1,1]$, which contradicts \textbf{A1}. 
Thus $|N(v')\cap A|<\gamma n$ and $|B|>n/5$ follows.

\begin{figure}[ht]
    \centering
    \includegraphics[width=0.2\linewidth]{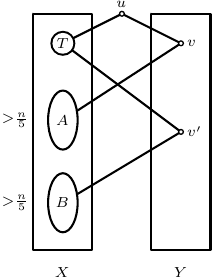}
    \caption{The structure of $A$ and $B$.}
\end{figure}

\begin{claim}\label{cl:size of XY}
    $|X|,|Y|>2n/5$.
\end{claim}

\begin{poc}
   Since $A\cap B=\varnothing$ and $|A|,|B|>n/5$, it follows that $|X|>2n/5$.
    If $| Y|\le 2n/5$, then for every $x\in A\cup B$, $|N(x)\cap Y|\ge (1/2+\varepsilon) |Y|$.
    Applying the Pairing Lemma (Lemma \ref{lmm:pairing}) to $(A, B, Y)$, there is a copy of $P_3[2h]$ with blocks $I\subseteq A,J\subseteq Y$ and $K\subseteq B$. Then $(I,J\setminus \{v,v'\},K, v',T,u,v)$ contains a copy $C_{7}[h,h,1,1,\gamma n,1,1]$, contradicting \textbf{A2}.
\end{poc}

By Claim \ref{cl:size of XY}, we have $|X|< 3n/5$ and  $|X\setminus(A\cup B)|<n/5$.
Thus, for every $y\in Y$, either $|N(y)\cap A|\ge\varepsilon n/2$ or $|N(y)\cap B|\ge\varepsilon n/2$.
We partition $Y$ into $Y_A$, $Y_B$, $Y'$ defined as follows:
\begin{align*}
    Y_A&=\{y\in Y: |N(y)\cap A|\ge \varepsilon n/2 \text{ and }|N(y)\cap B|< \varepsilon n/2\},\\
    Y_B&=\{y\in Y: |N(y)\cap B|\ge \varepsilon n/2 \text{ and }|N(y)\cap A|< \varepsilon n/2\},\\
    Y'&=Y\setminus (Y_A\cup Y_B).
\end{align*}
Clearly $v\in Y_A$ and $v'\in Y_B$. 

\begin{claim}\label{claim-5.14}
    $|Y'|<\varepsilon n/6$.
\end{claim}

\begin{poc}
Suppose, for a contradiction, that $|Y'|\geq \varepsilon n/6$.  
By the definition of $Y'$, for each $y\in Y'$, we have both $|N(y)\cap A|\ge \varepsilon n/2$ and $|N(y)\cap B| \ge \varepsilon n/2$. 
By applying Lemma \ref{lem-intersection} to $G[Y',A]$, we obtain a copy  $K_{2h/\varepsilon, \gamma n}$ with $W'\subseteq Y'$ and $A'\subseteq A$. Applying Lemma \ref{lem-intersection} to $G[W',B]$,  we obtain a copy  $K_{h, \gamma n}$ with $W\subseteq W'$ and $B'\subseteq B$. Then $(A',W,B',v',T,u,v)$ contains a copy of $C_7[h,h,1,1,\gamma n,1,1]$, which contradicts \textbf{A2}.
\end{poc}

Since $|N(x)\cap Y|\ge (1/5+\varepsilon)n-1$ for every $x\in X$ and $|Y'|<\varepsilon n/6$, by Claim \ref{claim-5.14}, we infer that either $|N(x)\cap Y_A|\ge \varepsilon n/2$ or $|N(x)\cap Y_B|\ge\varepsilon n/2$ holds. 
We partition $X$ into $X_A, X_B,$ and $X'$, which are defined as follows:
\begin{align*}
    X_A&=\{x\in X: |N(x)\cap Y_A|\ge \varepsilon n/2 \text{ and }|N(x)\cap Y_B|< \varepsilon n/2\},\\
    X_B&=\{x\in X: |N(x)\cap Y_B|\ge \varepsilon n/2 \text{ and }|N(x)\cap Y_A|< \varepsilon n/2\},\\
    X'&=X\setminus (X_A\cup X_B).
\end{align*}

\begin{claim}
$|A\setminus X_A|, |B\setminus  X_B|\leq \gamma n$, $n/5\leq |Y_A|, |Y_B|\leq 2n/5$ and $U_X\subseteq X_B$, $U_Y\subseteq Y_A$.
\end{claim}

\begin{poc}
First, we show that $|A\setminus X_A|\leq \gamma n$. Suppose for the sake of contradiction that $|A\setminus X_A|> \gamma n$.
Then for any $x\in A\setminus X_A$ we have $|N(x)\cap Y_B|\geq \varepsilon n/2$. 
Applying Lemma \ref{lem-intersection} to $G[A\setminus X_A, Y_B]$, we obtain a copy of $K_{h,\gamma n}$ with parts $I\subseteq A\setminus X_A$  and $C\subseteq Y_B$. Note that for each $y \in C$, by the definition of $Y_B$ we have $|N(y) \cap B| \ge \varepsilon n/2$. Thus,
\[
e(C,B) \geq |C|\cdot \frac{\varepsilon n}{2}\geq \frac{\varepsilon \gamma n^2}{2}. 
\]
By Theorem \ref{thm:KST}, there exists a copy of $K_{h,h}$ with parts in $C$ and $B$. 
Combining this with  $(v',T,v,u)$, we obtain a copy of $C_7[h,h,1,1,\gamma n,1,1]$, which contradicts \textbf{A2}. Thus $|A\setminus X_A|\leq \gamma n$. By a similar argument, one can show that $|B\setminus  X_B|\leq \gamma n$. 

Next, suppose that there exists $x \in U_X \setminus X_B$. 
Then $x$ has at least $\varepsilon n/2$ neighbors in $Y_A$. Let $C=N(x)\cap Y_A$. 
Since $|N(y)\cap A|>\varepsilon n/2$ for any $y\in C$, we can find a $K_{h,h}$ with parts $I\subseteq A$ and $J\subseteq C$. 
It follows that $(v,u,x,J,I)$ forms a copy of $C_5[1,1,1,h,h]$, which contradicts \textbf{A3}. Thus $U_X\subseteq X_B$. 

Note that $|A\setminus X_A|\leq \gamma n$ and $|A|>n/5$ imply $X_A\neq  \varnothing$. 
It follows that $|Y_A|\ge (1/5+\varepsilon)n-\varepsilon n/2-\varepsilon n/6>n/5$. 
By Claim \ref{cl:size of XY},  $|Y_B|\le|Y|-|Y_A|\leq 3n/5-n/5=2n/5$.
Similarly, we have $X_B \neq  \varnothing$, $|Y_B|\ge n/5$ and $|Y_A|\leq 2n/5$.

Finally, suppose that there exists $y \in U_Y \setminus Y_A$. Note that $|N(y)\cap B|\geq \varepsilon n/2$ and $|B\setminus X_B|\leq \gamma n$ imply $|N(y)\cap X_B|\geq \varepsilon n/3$. Let $S\subseteq N(y)\cap X_B\setminus T$ with $|S|=\gamma n$.   
Since $S \subseteq X_B$ and $T \subseteq U_X \subseteq X_B$, for any $x\in S\cup T$ we have
\[
|N(x)\cap Y_B|\geq \frac{n}{5}+\varepsilon n -\frac{\varepsilon n}{2}-\frac{\varepsilon n}{6} =\frac{n}{5}+\frac{\varepsilon n}{3} \geq \left(\frac{1}{2}+\frac{5\varepsilon }{6}\right) |Y_B|.
\]
Applying the Pairing Lemma (Lemma \ref{lmm:pairing}) to $(S,T,Y_B)$, one can find a  copy of $P_3[h]$. Together with $u$ and $y$, we get a copy of $C_5[1,1,h,h,h]$, which contradicts \textbf{A3}. Thus $U_Y\subseteq Y_A$.
\end{poc}

\begin{claim}\label{cl:YC}
    $Y'=\varnothing$.
\end{claim}

\begin{poc}
Suppose, for a contradiction, that there exists  $y\in Y'$. Then $|N(y)\cap A|\ge \varepsilon n/2$
and $|N(y)\cap B|\ge \varepsilon n/2$. Since $|A\setminus X_A|\leq \gamma n$ and  $|B\setminus  X_B|\leq \gamma n$, we infer that $|N(y)\cap X_A|>\varepsilon n/3$
and $|N(y)\cap X_B|>\varepsilon n/3$. 
Let $A'\subseteq N(y)\cap X_A$ and $B'\subseteq N(y)\cap X_B$ with $|A'|=|B'|=\varepsilon n/3$. 
Since $U_X, B'\subseteq X_B$ and $|Y'|<\varepsilon n/6$, for each $x\in U_X\cup B'$ we have
\[
|N(x) \cap Y_B| > \left(\frac{1}{5} + \varepsilon\right)n - \frac{\varepsilon n}{2} - \frac{\varepsilon n}{6} = \left(\frac{1}{5} + \frac{\varepsilon}{3}\right)n.
\]
Note that $|Y_B|\le 2n/5 $. Hence, for every $x\in  U_X\cup B'$, we have $|N(x)\cap Y_B|>(1/2+5\varepsilon/6)|Y_B|$.
Apply the Pairing Lemma (Lemma \ref{lmm:pairing}) to the triple $(U_X\setminus B', B', Y_B)$, we obtain a $P_3[h]$ with blocks $I\subseteq U_X\setminus B'$, $J\subseteq Y_B$ and $K\subseteq B'$. Then $(I,J,K,y,A',v,u)$ forms a copy $C_7[h, h, h, 1, \varepsilon n/3, 1, 1]$, as illustrated in Figure \ref{fig:both} (a), contradicting \textbf{A2}. 
\end{poc}

\begin{figure}[ht]
  \centering
  \begin{subfigure}[b]{.3\linewidth}
  \centering
    \includegraphics[width=0.8\linewidth]{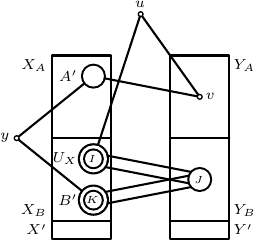}
    \caption{The case $y\in Y'$.}
    \label{sfig:YC}
  \end{subfigure}
  \begin{subfigure}[b]{.3\linewidth}
    \centering
    \includegraphics[width=0.8\linewidth]{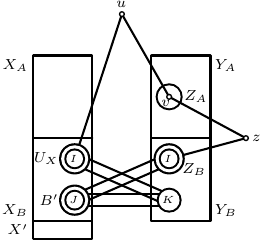}
    \caption{The case  $z\in X'$.}
    \label{sfig:XC}
  \end{subfigure}
   \begin{subfigure}[b]{.3\linewidth}
    \centering
    \includegraphics[width=0.8\linewidth]{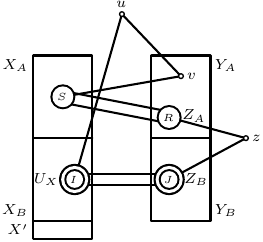}
    \caption{The case $z\in X'$.}
    \label{sfig:XC}
  \end{subfigure}
  \caption{The proofs of $Y'=\varnothing$ and $X'=\varnothing$.}
  \label{fig:both}
\end{figure}

\begin{claim}\label{cl:XC is empty}
    $X'=\varnothing$.
\end{claim}

\begin{poc}
Assume that $X'\neq \varnothing$ and $z\in X'$. Let 
\begin{align*}
    \text{$Z_A=N(z)\cap Y_A$\quad and \quad $Z_B=N(z)\cap Y_B$. }
\end{align*}

Let us show that $Z_A\cap U_Y=\varnothing$. Note that $|B\setminus X_B|\leq \gamma n$ implies $|N(y)\cap X_B|\geq \frac{\varepsilon n}{3}$ for every $y\in Z_B$. 
By Lemma \ref{lem-intersection}, there exists a copy of $K_{h,\gamma n}$ with parts $I\subseteq Z_B$ and $B'\subseteq X_B$, where $|I|=h$ and $|B'|=\gamma n$. 
Since $|U_X|\geq \frac{n}{10}$ and $|Y_B|\leq \frac{2n}{5}$, for every $x\in U_X\cup B'$, $|N(x)\cap (Y_B\setminus I)|\geq \frac{n}{5} +\frac{\varepsilon n}{3} \geq (\frac{1}{2}+\frac{5\varepsilon}{6})|Y_B\setminus I|$. 
Applying the Pairing Lemma (Lemma \ref{lmm:pairing}) to $(U_X\setminus B', B', Y_B\setminus I)$, we obtain a copy of $P_3[h]$ with blocks $J\subseteq B'$ and $K\subseteq Y_B\setminus I$ and $L\subseteq U_X\setminus B'$. If $U_Y\cap Z_A\neq \varnothing$, letting $v\in U_Y\cap Z_A$, then $(u,v,z,I,J,K,L)$ forms a copy of $C_7[h,h,h,h,1,1,1]$ as shown in Figure \ref{fig:both} (b), a contradiction. Thus $Z_A\cap U_Y=\varnothing$.

Next we show that there exist $x_0\in U_X$ and $y_0\in Z_B$ such that $|N(x_0)\cap Z_B|\leq \frac{\varepsilon n}{3}$ and $|N(y_0)\cap U_X|\leq \frac{\varepsilon n}{3}$. 
Indeed, if $|N(x)\cap Z_B|\geq \frac{\varepsilon n}{3}$ for every $x\in U_X$ or $|N(y)\cap U_X|\geq \frac{\varepsilon n}{3}$ for every $y\in Z_B$, then  by Theorem \ref{thm:KST} there is a copy of $K_{h,h}$ with parts $I\subseteq U_X$ and $J\subseteq Z_B$. Let $v\in U_Y$. 
Since $v\in Y_A$ and $Z_A\subseteq Y_A$, for any $y\in Z_A$ we have $|N(y)\cap N(v)|\geq 2(\frac{n}{5}+\frac{\varepsilon n}{2}) -|X\setminus B|\geq \varepsilon n$. 
By Lemma \ref{lem-intersection}, there exists $R\subseteq Z_A$ with $|R|=h$ such that $|\bigcap_{y\in R}N(y)\cap N(v)|\geq \gamma n$. 
Let $S=\big(\bigcap_{y\in R}N(y)\big)\cap N(v)$. Then $(u,v,S,R,z,J,I)$ forms a copy of $C_7[1,1,\gamma n, h,1,h,h]$ as shown in Figure \ref{fig:both} (c), a contradiction. Thus,  there exist $x_0\in U_X$ and $y_0\in Z_B$ such that $|N(x_0)\cap Z_B|\leq \frac{\varepsilon n}{3}$ and $|N(y_0)\cap U_X|\leq \frac{\varepsilon n}{3}$. 

Now we have
\begin{align*}
|Y|+|X_B|
&\geq  \Big(|U_Y|+|Z_A|+|Z_B|+|N(x_0)\cap Y_B\setminus Z_B|\Big) + \Big(|U_X|+|N(y_0)\cap X_B\setminus U_X|\Big)\\[3pt]
&= \Big(|U_X|+|U_Y|\Big)+\Big(|Z_A|+|Z_B|\Big)+\Big(|N(x_0)\cap Y_B\setminus Z_B|+|N(y_0)\cap X_B\setminus U_X|\Big)\\[3pt]
&\geq 2\left(\frac{1}{5}+\varepsilon\right)n+  \frac{2}{5}n\\[3pt]
&>\frac{4n}{5}+2\varepsilon n. 
\end{align*}
Moreover, $|X_A|\geq |A|-|A\setminus X_A| \ge n/5-\gamma n$. It follows that
$
|X|+|Y|\geq n+2\varepsilon n-\gamma n>n,
$
a contradiction. 
\end{poc}

\begin{figure}[ht]
  \centering
  \begin{subfigure}[b]{0.33\linewidth}
    \centering
    \includegraphics[width=0.8\linewidth]{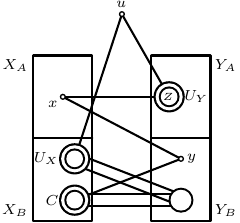}
    \caption{The case $|N(x)\cap U_Y|\geq \gamma n$.}
    \label{sfig:Nx-UY}
  \end{subfigure}\hspace{10pt}
  \begin{subfigure}[b]{0.30\linewidth}
    \centering
    \includegraphics[width=0.8\linewidth]{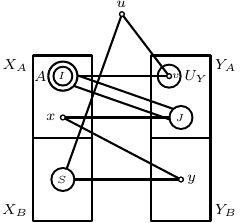}
    \caption{ $xy\in E(X_A,Y_B)$.}
    \label{sfig:C7-proof}
  \end{subfigure}\hspace{10pt}
    \begin{subfigure}[b]{0.30\linewidth}
    \centering
    \includegraphics[width=0.8\linewidth]{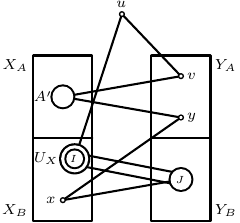}
    \caption{$xy\in E(X_B,Y_A)$.}
    \label{sfig:C7-proof}
  \end{subfigure}
  \caption{Proofs of Claims \ref{cl:XA} and \ref{cl:YA}.}
  \label{fig:Claim-6.7}
\end{figure}

\begin{claim}\label{cl:XA}
  $e(X_A,Y_B)=0$.
\end{claim}

\begin{poc}
    Suppose, for the sake of contradiction, that there exists an edge $xy$ with $x \in X_A$ and $y \in Y_B$.
    
   First we assert that $|N(x)\cap U_Y|< \gamma n$. Indeed, otherwise let $Z\subseteq  N(x)\cap U_Y$ with $|Z|=\gamma n$.  Since $y\in Y_B$ and $|B\setminus X_B|\leq \gamma n$,  $|N(y)\cap X_B|\geq \varepsilon n/3$. 
   Let $C$ be a subset of $N(y)\cap X_B$ with $|C|= \varepsilon n/3$. 
    By Claim \ref{cl:size of XY}, $|Y_B|\le 2n/5$. Thus, for any $w\in C\cup U_X$,  
    \[
    |N(w)\cap Y_B|\ge (1/5+\varepsilon)n-\varepsilon n/2\ge (1/2+\varepsilon)|Y_B|.
    \]
    Now apply  the Pairing Lemma (Lemma \ref{lmm:pairing}) to $(C,U_X\setminus C,Y_B)$, one can find a copy of $P_3[h]$. Together with $(u,Z,x,y)$, we get a copy of  $C_7[h,h,h,1,1,\gamma n,1]$, which contradicts \textbf{A2}. Thus $|N(x)\cap U_Y|<\gamma n$.

Since $|N(x)\cap U_Y|<\gamma n$ and $|N(x)\cap Y_A|\geq (1/5+\varepsilon/2)n$, we infer that  
\[
|Y_A|> |N(x)\cap Y_A|+|U_Y|-\gamma n>n/5+|U_Y|.
\]
As $ X_B\neq \varnothing$, we infer that  $|Y_B|> n/5$. 
    Thus, 
\begin{align*}
    |X\setminus A|&=n-1-|Y_A|-|Y_B|-|A|< n-(n/5+|U_Y|)-n/5-n/5= 2n/5-|U_Y|.
\end{align*}
Note that $|A\setminus X_A|<\gamma n$  and $U_X\subseteq X_B$ imply $|U_X\cap (X\setminus A)|>n/5-|U_Y|-\gamma n$. 
Since $|N(y)\cap (X\setminus A)|>(1/5+\varepsilon/2)n$, we infer that 
$$|N(y)\cap N(u)\cap (X\setminus A)|=|N(y)\cap (X\setminus A)|+|U_X\cap (X\setminus A)|-|X\setminus A|>\varepsilon n/3.$$ 
Let $S\subseteq  N(y)\cap N(u)\cap (X\setminus A)$ with $|S|=\varepsilon n/3$. 

Since $x\in X_A$, $|N(x)\cap Y_A|\geq \varepsilon n/2$. Moreover, for every $y'\in N(x)\cap Y_A$ we have $|N(y')\cap A|\geq \varepsilon n/2$. Thus by Theorem \ref{thm:KST} one can find a copy of $K_{2h,2h}$ with parts $I\subseteq A$ and $J\subseteq N(x)\cap Y_A\setminus \{v\}$. Then $(x,y,S,u,v,I,J)$ contains a copy of $C_7[1,1,\gamma n,1,1,h,h]$, which contradicts \textbf{A2}. Thus $e(X_A,Y_B)=0$.
\end{poc}

\begin{claim}\label{cl:YA}
  $e(X_B,Y_A)=0$.
\end{claim}

\begin{poc}
Suppose that there is an edge $xy$ with $x\in X_B$ and $y\in Y_A$.
Then $|N(y)\cap A|\geq \varepsilon n/2$. Let $A'\subseteq N(y)\cap A$ with $|A'|=\gamma n$. Since $U_X\subseteq X_B$ and $|Y_B|\leq 2n/5$, for any $x'\in U_X\setminus \{x\}$ we have $|N(x)\cap N(x')\cap Y_B|\geq \varepsilon n$. Thus by Theorem \ref{thm:KST} there exists a $K_{h,h}$ with parts $I\subseteq U_X\setminus \{x\}$ and $J\subseteq N(x)\cap Y_B$. Now $(x,y,A',v,u, I,J)$ forms a copy of $C_7[1,1,\gamma n,1,1,h,h]$, which contradicts \textbf{A2}. 
\end{poc}

Now $(X_A\cup Y_B\cup \{u\}, X_B\cup Y_A)$ forms a bipartition of $G$ and the lemma is proven. 
\end{proof}

\section{Proof of Theorem \ref{thm-main2} (iii)}\label{Section:5}

We need the following two graphs $\Gamma(h)$ and $\Gamma'(h)$ as shown in Figure \ref{fig:gamma}. 
Let $a,b,c,d$ be distinct vertices and let $A,B,C,X,Y,Z,U,W$ be disjoints vertex sets with each of size $h$. 

Define $\Gamma(h)$ as a graph on the vertex set $A\cup B\cup X\cup Y\cup Z\cup U\cup W\cup \{a,b,c,d\}$ such that $(a,b,c,U,W,d,X,Y,Z)$ forms a copy of $C_9[1,1,1,h,h,1,h,h,h]$ and $(a,b,B,A)$  forms a copy of $C_4[1,1,h,h]$, as shown in Figure \ref{fig:gamma} (a). 

Define $\Gamma'(h)$ as a graph on the vertex set $B\cup C\cup X\cup Y\cup Z\cup U\cup W\cup \{a,b,c,d\}$ such that $(a,b,c,U,W,d,X,Y,Z)$ forms a copy of $C_9[1,1,1,h,h,1,h,h,h]$ and $(b,c,C,B)$  forms a copy of $C_4[1,1,h,h]$, as shown in Figure \ref{fig:gamma} (b).

\begin{figure}[ht]
  \centering
  \begin{subfigure}[b]{.4\linewidth}
  \centering
    \includegraphics[width=0.77\linewidth]{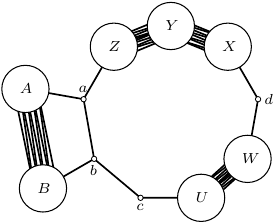}
   \caption{$\Gamma(h)$}
  \end{subfigure}
  \begin{subfigure}[b]{.4\linewidth}
    \centering
    \includegraphics[width=0.7\linewidth]{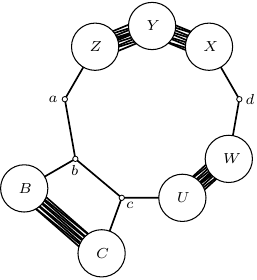}
    \caption{$\Gamma'(h)$}
  \end{subfigure}
  \caption{The graphs $\Gamma(h)$ and $\Gamma'(h)$.}
  \label{fig:gamma}
\end{figure}

\begin{lemma}\label{lem-finalkey}
  Let $H$ be a color-critical graph with $\chi(H)=3$ and $|H|=h$.
  Let $G$ be an $H$-free $n$-vertex graph with $\delta(G)\geq cn$. If $H\subseteq G_1(4h)$ and $H\subseteq G_5(5h+4)$, then $G$ is both $\Gamma(t)$-free and $\Gamma'(t)$-free  with $t=\lceil\frac{2h}{c}\rceil$. 
\end{lemma}

\begin{proof}
Since $H\subseteq G_1(4h)$, we infer that there is a vertex $x\in V(H)$ with degree 2. 
Let $y,z$ be neighbors of $x$ in $H$. Let $(C,v,u,A,B,A',u',v',C')$ form a copy of $C_9[h,1,1,h,h,h,1,1,h]$, i.e. $G_5(5h+4)$, and let  $\phi: V(H) \to V(G_5(5h+4))$ be an embedding from $H$ to $G_5(5h+4)$.

Let $a=\phi^{-1}(u)$, $b=\phi^{-1}(v)$, $a'=\phi^{-1}(u')$, $b'=\phi^{-1}(v')$.  Since $uv,u'v'$ are color-critical edges of $G_5(5h+4)$, we infer that $ab,a'b'$ are color-critical edges of $H$. 

By symmetry, we distinguish four cases.

\begin{mycase}{Case 1.}
    $\phi(x)=u$.
\end{mycase}

Note the edge $uv$ must exist in $H$.
By symmetry, we may assume $\phi(y)=v$ and $\phi(z)\in A$. Let
\[
W=\{w\in V(H)\colon  w\neq z,\ \phi(w)\in A\},\ W'=\{w\in V(H)\colon \phi(w)\in B\}.
\]
Since $d_H(x)=2$, we infer that $N_H(w)\subseteq W'$ for each $w\in W$. Then $H-W$ is contained in a copy of $C_9[h,1,1,1,h,h,1,1,h]$ as shown in Figure \ref{fig:gamma2} (a). By Lemma \ref{lem-5.1}, $G$ is  $C_9[h,1,1,1,t,h,1,1,h]$-free. Since $C_9[h,1,1,1,t,h,1,1,h]$ is contained in both $\Gamma(t)$ and $\Gamma'(t)$, we infer that   $G$ is both $\Gamma(t)$-free and $\Gamma'(t)$-free.

\begin{figure}[ht]
  \centering
  \begin{subfigure}[b]{.4\linewidth}
  \centering
    \includegraphics[width=0.77\linewidth]{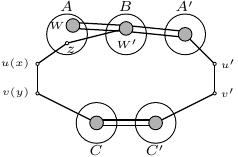}
   \caption{Case 1}
  \end{subfigure}
  \begin{subfigure}[b]{.4\linewidth}
    \centering
    \includegraphics[width=0.77\linewidth]{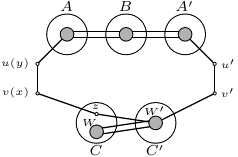}
    \caption{Case 2}
  \end{subfigure}
  \caption{ The proof of Case 1 and Case 2.}
  \label{fig:gamma2}
\end{figure}

\begin{mycase}{Case 2.}
    $\phi(x)=v$.
\end{mycase}

By symmetry, we may assume $\phi(y)=u$ and $\phi(z)\in C$. Let
\[
W=\{w\in V(H)\colon  w\neq z,\ \phi(w)\in C\},\ W'=\{w\in V(H)\colon \phi(w)\in C'\}.
\]
Since $d_H(x)=2$, we infer that $N_H(w)\subseteq W'$ for each $w\in W$. Then $H-W$ is contained in a copy of $C_9[1,1,1,h,h,h,1,1,h]$  as shown in Figure \ref{fig:gamma2} (b). By Lemma \ref{lem-5.1},  $G$ is  $C_9[1,1,1,1,h,h,1,1,t]$-free. Since $C_9[1,1,1,1,h,h,1,1,t]$ is contained in  both $\Gamma(t)$ and $\Gamma'(t)$, we infer that $G$ is both $\Gamma(t)$-free and $\Gamma'(t)$-free.

\begin{mycase}{Case 3.}
    $\phi(x)\in A\cup B$.
\end{mycase}

Recall that $ab$ and $a'b'$ are both color-critical edges of $H$. 
Since $H-x$ is bipartite, every odd cycle of $H$ passes through $x$.
Moreover, since $d_H(x)=2$ and every $y-z$ path in $H-x$ has the same parity, it follows that $x$ cannot lie on an even cycle of $H$.
As $x\notin \{a,b,a',b'\}$, we infer that there exist paths from $y$ to $a$ and from $z$ to $a'$. Note that the connected component containing $x$  of $H-\{a,a'\}$ is bipartite.  It follows that no cycle in the connected component of $x$ in $H-\{a,a'\}$ passes through $x$.

If $y=a=\phi^{-1}(u)$, then $z\neq a'$ and $H$ is contained in a graph as shown in Figure \ref{fig:gamma3} (a). (We simplify the setting here, since in $H-ab-a'b'$ we may always fold to place vertices at odd distance from $y$  into part $A$  and those at even distance from $y$  into part $B$. Henceforth, we shall only depict the simplified configuration.)
By the minimum degree condition of $G$ and Lemma \ref{lem-5.1},  $G$ is $C_9[1,1,1,1,t,1,1,h,h]$-free and therefore $\Gamma(h)$-free and $\Gamma'(t)$-free. The case $z=a'=\phi^{-1}(u')$ is symmetric to the case $y=a$. Thus we may assume that $y\neq a$ and $z\neq a'$. Then $H$ is contained in a graph as shown in Figure  \ref{fig:gamma3} (b).  It can be checked that  $H$ is contained in  both $\Gamma(t)$ and $\Gamma'(t)$. Thus $G$ is both $\Gamma(t)$-free and $\Gamma'(t)$-free.

\begin{figure}[ht]
  \centering
  \begin{subfigure}[b]{.4\linewidth}
  \centering
    \includegraphics[width=0.77\linewidth]{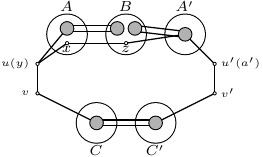}
   \caption{The case $y=a$.}
  \end{subfigure}
  \begin{subfigure}[b]{.4\linewidth}
    \centering
    \includegraphics[width=0.77\linewidth]{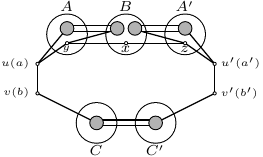}
    \caption{The case $y\neq a$ and $z\neq b$.}
  \end{subfigure}
  \caption{ The proof of Case 3.}
  \label{fig:gamma3}
\end{figure}

\begin{mycase}{Case 4.}
   $\phi(x)\in C$.
\end{mycase}

Since every odd cycle in $H$ passes through $x$ and $x\notin \{a,b,a',b'\}$, we infer that there exist paths from $y$ to $b$ and from $z$ to $b'$. Note that the connected component of $x$ in $H-\{b,b'\}$ is bipartite. 
Since $x$ does not lie on any even cycle, we infer that  no cycle in the connected component of $x$ in $H-\{b,b'\}$ passes through $x$.

If $y=b=\phi^{-1}(v)$, then $H$ is contained in the graph as shown in Figure \ref{fig:gamma4} (a). 
By the minimum degree condition and Lemma \ref{lem-5.1},  we infer that $G$ is $\Gamma_1$-free, where 
$\Gamma_1$ is the graph with vertex set $\{u_1,u_2,u_3,u_4,u_5,u_6\}\cup U_7\cup U_8\cup U_9\cup U_4\cup U_5$ such that $(u_1,u_2,u_3,u_4,u_5,u_6,U_7,U_8,U_9)$ forms a copy of $C_9[1,1,1,1,1,1,h,h,h]$ and $(u_4,u_5,U_5,U_4)$ forms a copy of $C_4[1,1,h,h]$. It is easy to check that $\Gamma_1\subseteq \Gamma(t)$ and $\Gamma_1\subseteq \Gamma'(t)$. Therefore, $G$ is both $\Gamma(t)$-free and $\Gamma'(t)$-free.

\begin{figure}[ht]
  \centering
  \begin{subfigure}[b]{.4\linewidth}
  \centering
    \includegraphics[width=0.77\linewidth]{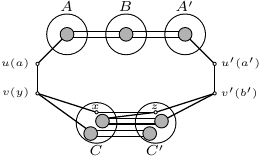}
   \caption{The case $y=b$}
  \end{subfigure}
  \begin{subfigure}[b]{.4\linewidth}
    \centering
    \includegraphics[width=0.77\linewidth]{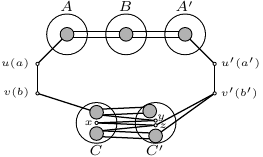}
    \caption{The case $y,z\in C'$.}
  \end{subfigure}
  \caption{ The proof of Case 4.}
  \label{fig:gamma4}
\end{figure}

If $y,z\in C'$, then $H$ is contained in the graph shown in Figure \ref{fig:gamma4} (b). By Lemma \ref{lem-5.1}, we infer that $G$ is $\Gamma_2$-free, where $\Gamma_2$ is the graph with vertex set $\{u_1,u_2,u_3,u_4,u_5,u_9,u_{10}\}\cup U_6\cup U_7\cup U_8\cup U_{11}\cup U_3\cup U_4$ such that 
\[
(u_1,u_2,u_3,u_4,u_5,U_6,U_7,U_8,u_9,u_{10},U_{11})\mbox{ forms a copy of } C_{11}[1,1,1,1,1,h,h,h,1,1,t]
\]
and $(u_3,u_4,U_4,U_3)$ forms a copy of $C_4[1,1,h,h]$.
It is easy to check that $\Gamma_2\subseteq \Gamma(t)$ and $\Gamma_2\subseteq \Gamma'(t)$. Thus $G$ is  both $\Gamma(t)$-free and $\Gamma'(t)$-free.  
\end{proof}

\begin{proof}[\textbf{Proof of Theorem \ref{thm-main2} (iii)}]
The construction \(G_6(n)\), together with \eqref{lower-bdd}, shows that \(\delta_{\mathrm{ext}}(H,2)\ge 1/6\) whenever \(g_{\mathrm{odd}}(H)\ge 5\).
It remains to establish the matching upper bound, namely that \(\delta_{\mathrm{ext}}(H,2)\le 1/6\) whenever \(H \hookrightarrow \mathcal{G}_{[5]}\).

Let $G$ be an $H$-free graph with $\delta(G)\geq (1/6+\varepsilon)n$. Fix some  $\gamma>0$ such that $\gamma  \ll \varepsilon$. 
Since $H\subseteq G_1(4h+4)$, by Lemma \ref{lem-5.1}, we have 
\begin{itemize}
    \item [\rm \textbf{B1}] $G$ is $C_5[1,1,1,h,h]$-free and $C_7[1,1,1,h,h,h,h]$-free.
\end{itemize}
Since $H\subseteq G_i(4h+4)$ for $i=1,2,3,4$, by Lemma \ref{lem-4.4} we infer that  
\begin{itemize}
    \item [\rm \textbf{B2}] $G$ is $C_5[12h,1,12h,1,1]$-free;
    \item [\rm \textbf{B3}] $G$ is $C_7[h,h,1,1,12h,1,1]$-free.
\end{itemize}
Moreover, since $H\subseteq G_5(5h+4)$, we infer that $g_{\mathrm{odd}}(H)\geq 9$ and 
\begin{itemize}
    \item [\rm \textbf{B4}] $G$ is $C_9[h,h,h,1,1,h,h,1,1]$-free.
\end{itemize}
Suppose that $G-u$ is bipartite. 
    We need to show that $G$ is bipartite.
    Let $X\cup Y$ be the bipartition of $G-u$ with $|X|\le |Y|$.
    Let 
    $$\text{$N(u)\cap X=U_X$\quad and \quad $N(u)\cap Y=U_Y$.}$$
    Assume that $U_X,U_Y\neq \varnothing$.

\begin{mycase}{Case 1.}
    $|U_X|\ge \varepsilon n$.
\end{mycase}

\begin{figure}[ht]
    \centering
    \includegraphics[width=0.24\linewidth]{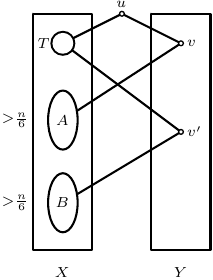}
    \caption{The structure of $A$ and $B$.}
     \label{fig:thm1-11iv-case1}
\end{figure}

Fix a vertex $v\in U_Y$. Since $|U_X|\geq \varepsilon n$ and $\delta(G)\geq (1/6+\varepsilon )n$, each $x\in U_X$ has at least  $n/6$ neighbors in $Y\setminus \{v\}$. 
    By an averaging argument, there exists a vertex $v'\in Y\setminus \{v\}$ such that 
  \[
  |N(v')\cap U_X| \geq \frac{1}{|Y\setminus \{v\}|}\cdot |U_X| \cdot \frac{n}{6}\geq \frac{1}{n}\cdot \varepsilon n\cdot  \frac{n}{6}=\frac{\varepsilon n}{6}\geq \gamma n.
  \]
Let $T\subseteq N(v')\cap U_X$ with $|T|=\gamma n$ and let $A=(N(v)\cap X)\setminus T$, $B=(N(v')\cap X)\setminus (A\cup T)$. Clearly $|A|>n/6$. If $|N(v')\cap A|\geq \gamma n$, then $(v,u,T,v',N(v')\cap A)$ contains a copy of $C_5[1,1,\gamma n, 1, \gamma n]$, which contradicts \textbf{B2}.
Thus $|N(v')\cap A|< \gamma n$ and $|B|> n/6$ follows (see Figure \ref{fig:thm1-11iv-case1}).

  Note that $|X|\le n/2$. It follows that  $|X\setminus(A\cup B)|\le n/6$. Thus for every $y\in Y$, either $|N(y)\cap A|\ge\varepsilon n/2$ or $|N(y)\cap B|\ge\varepsilon n/2$.
  Partition $Y$ into $Y_A\cup Y_B\cup Y'$, where
\begin{align*}
    Y_A&=\{y\in Y: |N(y)\cap A|\ge \varepsilon n/2 \text{ and }|N(y)\cap B|< \varepsilon n/2\},\\
    Y_B&=\{y\in Y: |N(y)\cap B|\ge \varepsilon n/2 \text{ and }|N(y)\cap A|< \varepsilon n/2\},\\
    Y'&= Y\setminus (Y_A\cup Y_B).
\end{align*}
Clearly $v\in Y_A$ and $v'\in Y_B$.

\begin{claim}
    $|Y'|< \varepsilon n/4$.
\end{claim}

\begin{poc}
Indeed, if  $|Y'|\geq \varepsilon n/4$, then by applying Lemma \ref{lem-intersection} twice one can find a copy of $P_3[h]$ with blocks $I\subseteq A, J\subseteq Y', K\subseteq B$. Then $(u,v,I,J,K,v',T)$ forms a copy of $C_7[1,1,h,h,h,1,\gamma n]$, which contradicts \textbf{B3}. 
\end{poc}

Since $|Y'|<\varepsilon n/4$, we infer that either $|N(x)\cap Y_A|\geq \varepsilon n/2$ or $|N(x)\cap Y_B|\geq \varepsilon n/2$ for all $x\in X$. Partition $X$ into $X_A\cup X_B\cup X'$, where 
\begin{align*}
    X_A&=\{x\in X: |N(x)\cap Y_A|\ge \varepsilon n/2 \text{ and }|N(x)\cap Y_B|< \varepsilon n/2\},\\
    X_B&=\{x\in X: |N(x)\cap Y_B|\ge \varepsilon n/2 \text{ and }|N(x)\cap Y_A|< \varepsilon n/2\},\\
    X'&=X\setminus (X_A\cup X_B). 
\end{align*}

\begin{claim}\label{claim-6.3}
    $U_X\subseteq X_B$ and $U_Y\subseteq Y_A$. 
\end{claim}

\begin{poc}
If there exists some $x\in U_X\cap (X_A\cup X')$, then $|N(x)\cap Y_A|\geq \frac{\varepsilon n}{2}$ and for any $y\in N(x)\cap Y_A$, $|N(y)\cap A|\geq \frac{\varepsilon n}{2}$. by Theorem \ref{thm:KST} one can find a $K_{h,h}$ with parts $I\subseteq N(x)\cap Y_A$ and $J\subseteq A$. Then $(I,J,v,u,x)$ forms a copy of $C_5[h,h,1,1,1]$, contradicting the fact that $G$ is $H$-free and $H\subseteq G_1(4h)$.  Thus   $U_X\subseteq X_B$. 

Assume that  there is some $y\in U_Y\setminus Y_A$. Then $|N(y)\cap B|\geq \varepsilon n/2$. 
Choose $T'\subseteq N(y)\cap B$ with $|T'|=\gamma n$.
It follows that $(u,T,v',T', y)$ forms a copy of $C_5[1,\gamma n,1,\gamma n,1]$, which contradicts \textbf{B2}. Thus $U_Y\subseteq Y_A$. 
\end{poc}

\begin{figure}[ht]
  \centering
  \begin{subfigure}[b]{.3\linewidth}
  \centering
    \includegraphics[width=0.88\linewidth]{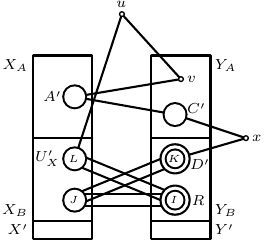}
    \caption{ The case $x\in X'$.}
  \end{subfigure}
  \begin{subfigure}[b]{.3\linewidth}
    \centering
    \includegraphics[width=0.8\linewidth]{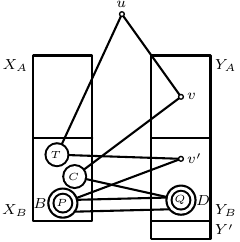}
    \caption{The case  $x\in A\cap  X_B$.}
  \end{subfigure}
    \begin{subfigure}[b]{.3\linewidth}
    \centering
    \includegraphics[width=0.8\linewidth]{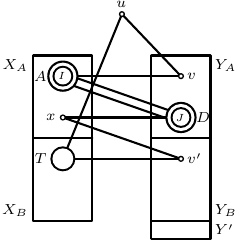}
    \caption{The case  $x\in B\cap X_A$.}
  \end{subfigure}
  \caption{Proof of claims}
  \label{fig:thm194-2-4}
\end{figure}

\begin{claim}\label{claim-6.4}
    $X'=\varnothing$.
\end{claim}

\begin{poc}
Assume $x\in X'$. By the definition of $X'$, we have $|N(x)\cap Y_A|\geq \varepsilon n/2$ and  $|N(x)\cap Y_B|\geq \varepsilon n/2$. 
Let $C\subseteq N(x)\cap Y_A$ and $D\subseteq N(x) \cap Y_B$ with $|C|=|D|= \varepsilon n/2$. 
Since $C \subseteq Y_A$, the vertices in $C$ have large degree into $A$. By Theorem \ref{thm:KST}, one can find a copy of $K_{h,h}$ with parts $A'\subseteq A$ and $C'\subseteq C$. 

Recall that $|U_X| \geq \varepsilon n$. By Claim \ref{claim-6.3}, we have $U_X \subseteq X_B$.
By Lemma \ref{lem-intersection} there is a $K_{h,\gamma n}$ with parts $L\subseteq U_X$ and $R\subseteq Y_B$.  Let $D'=D\setminus R$. 
Since $D'\cup R\subseteq Y_B$, $|N(y)\cap A|<\varepsilon n/2$ for each $y\in D'\cup R$. 
Since $|X| \leq n/2$ and $|A|>n/6$, we have $|X \setminus A| \leq n/3$.  
Then for every $w\in R\cup D'$, we have
\begin{align*}
    |N(w)\cap (X\setminus A)|\ge(1/6+\varepsilon)n-\varepsilon n/2>(1/2+\varepsilon)|X\setminus A|
\end{align*}
Apply the Pairing Lemma (Lemma \ref{lmm:pairing}) to $(R,D',X\setminus A)$ one can find a copy of $P_3[h]$ with blocks $I\subseteq R$, $J\subseteq X\setminus A$, $K\subseteq D'$. Now $(L,I,J,K,x,C',A',v,u)$ forms a copy of $C_9[h,h,h,h,1,h,h,1,1]$ as shown in Figure \ref{fig:thm194-2-4} (a), which contradicts \textbf{B4}. 
\end{poc}

\begin{claim}\label{cl:5.5}
$|A\setminus X_A|<\gamma n$ and $B\subseteq X_B$.
\end{claim}

\begin{poc}
Suppose that  $|A\setminus X_A|\geq \gamma n$. Then by Lemma \ref{lem-intersection}, there is a $K_{h,\gamma n}$ with parts $C\subseteq A\setminus X_A$ and $D\subseteq Y_B$. Since $|N(w)\cap B|\ge \varepsilon n/2$ for every $w\in D$, by Theorem \ref{thm:KST} one can find a $K_{h,h}$ with parts $P\subseteq B$ and $Q\subseteq D$. Then $(u,T,v',P,Q,C,v)$ forms a copy of $C_7[1,\gamma n,1,h,h,h,1]$ as shown in Figure \ref{fig:thm194-2-4} (b), which contradicts \textbf{B3}. Thus $|A\setminus X_A|<\gamma n$.

If there exists some $x\in B\cap X_A$, then let $D=N(x)\cap Y_A\setminus \{v\}$. By Theorem \ref{thm:KST}, one can find a $K_{h,h}$ with parts $I\subseteq A\cap X_A\setminus\{x\}$ and $J\subseteq D$. Then $(u,T,v',x,J,I,v)$ forms a copy of $C_7[1,\gamma n,1,1,h,h,1]$ as shown in Figure \ref{fig:thm194-2-4} (c), which contradicts \textbf{B3}. Thus $B\subseteq X_B$.
\end{poc}

\begin{figure}[ht]
  \centering
  \begin{subfigure}[b]{.3\linewidth}
  \centering
    \includegraphics[width=0.8\linewidth]{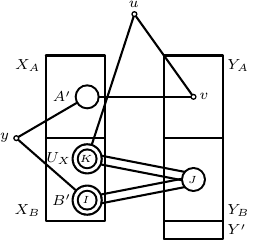}
    \caption{The case $|Y_B|\leq \frac{n}{3}$.}
  \end{subfigure}
  \begin{subfigure}[b]{.35\linewidth}
    \centering
    \includegraphics[width=0.7\linewidth]{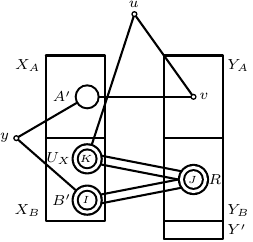}
    \caption{The case $|X\setminus A|\leq  \frac{n}{6}+|U_X|$.}
  \end{subfigure}
    \begin{subfigure}[b]{.3\linewidth}
    \centering
    \includegraphics[width=0.8\linewidth]{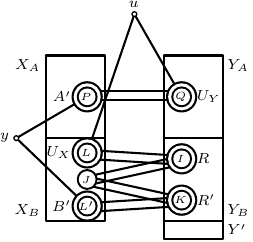}
    \caption{The third case.}
  \end{subfigure}
  \caption{Proof of $Y'=\varnothing$. }
  \label{fig:thm194-5-7}
\end{figure}

\begin{claim}
    $Y'=\varnothing$.
\end{claim}

\begin{poc}
Assume $y\in Y'$. Choose $A'\subseteq N(y)\cap A$ and $B'\subseteq N(y)\cap B$ with $|A'|=|B'|=\varepsilon n/2$. Note that  $|U_X| \geq \varepsilon n$. If $|Y_B|\leq n/3$, then 
for any $x\in X_B$ we have 
\[
|N(x)\cap Y_B|\geq \frac{n}{6}+\frac{\varepsilon n}{4} \geq \left(\frac{1}{2}+\frac{3\varepsilon }{4}\right)|Y_B|.
\]
Apply the Pairing Lemma (Lemma \ref{lmm:pairing}) to $(B', U_X\setminus B', Y_B)$, one can find a $P_3[h]$ with blocks $I\subseteq B'$, $J\subseteq Y_B$ and $K\subseteq U_X\setminus B'$. Then $(I,J,K,u,v,A',y)$ forms a copy of $C_7[h,h,h,1,1,\varepsilon n/2,1]$ as shown in Figure \ref{fig:thm194-5-7} (a), which contradicts \textbf{B3}. Thus we may assume $|Y_B|\geq n/3$. 

By Lemma \ref{lem-intersection}, there is a $K_{h,\gamma n}$ with parts $I\subseteq B'$ and $R\subseteq Y_B$.
If $|X\setminus A|\leq  n/6+|U_X|$, then $|N(y')\cap U_X|\geq \varepsilon n/2$ for any $y'\in R$. By Theorem \ref{thm:KST} there is a copy of $K_{2h,2h}$ with parts $J\subseteq R$ and $K\subseteq U_X$. Then $(I,J,K\setminus I,u,v,A',y)$ contains a copy of $C_7[h,h,h,1,1,\varepsilon n/2,1]$  as shown in Figure \ref{fig:thm194-5-7} (b), which contradicts \textbf{B3}. Thus we may assume $|X\setminus A|>  n/6+|U_X|$.

By $|Y_B|\geq n/3$, $|X\setminus A|>  n/6+|U_X|$ and $|U_X|+|U_Y|\geq n/6$, we infer that 
\[
|Y_A| \leq n-|X\setminus A|-|A| -|Y_B| \leq n-\Big(\frac{n}{6}+|U_X|\Big)-\frac{n}{6}-\frac{n}{3} =\frac{n}{3}-|U_X| \leq \frac{n}{6}+|U_Y|.
\]
Since $|Y'|<\varepsilon n/4$, we infer that $|N(x)\cap U_Y|\geq \varepsilon n/4$ for any $x\in A'\cap X_A$. Note that $|A\setminus X_A|<\gamma n$ and $A'\subseteq A$ implies $|A'\cap X_A|>\varepsilon n/3$. 
By Theorem \ref{thm:KST} there is a  $K_{h,h}$ with parts $P\subseteq A'\cap X_A$ and $Q\subseteq U_Y$.  
Since $|U_X| \geq \varepsilon n$ and $U_X\subseteq X_B$, by Lemma \ref{lem-intersection} there is a $K_{h,\gamma n}$ with $L\subseteq U_X$ and $R\subseteq Y_B$.  
Similarly, by  Lemma \ref{lem-intersection} there is a $K_{h,\gamma n}$ with parts $L'\subseteq B'\setminus L$ and $R'\subseteq Y_B \setminus  R$. 
Since $R\cup R'\subseteq Y_B$ and $|X\setminus A| \leq n/3$, for any $y'\in R\cup R'$ we have 
$$|N(y')\cap (X\setminus A)|\geq (1/6+\varepsilon/2)n> (1/2+\varepsilon)|X\setminus A|.$$
Applying the Pairing Lemma (Lemma \ref{lmm:pairing}) one can find a copy of $P_3[h]$ with blocks $I\subseteq R, J\subseteq X\setminus (A\cup L\cup L'), K\subseteq R'$. Then $(u,Q,P,y,L', K,J,I,L)$ forms a copy of $C_9[1,h,h, 1,h,h,h,h, h]$  as shown in Figure \ref{fig:thm194-5-7} (c), which contradicts \textbf{B4}. Thus $Y'=\varnothing$.  
\end{poc}

\begin{figure}[ht]
  \centering
  \begin{subfigure}[b]{.3\linewidth}
  \centering
    \includegraphics[width=0.77\linewidth]{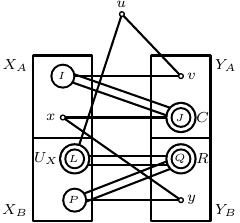}
    \caption{The case $E(X_A,Y_B)\neq \varnothing$.}
  \end{subfigure}
  \begin{subfigure}[b]{.34\linewidth}
    \centering
    \includegraphics[width=0.7\linewidth]{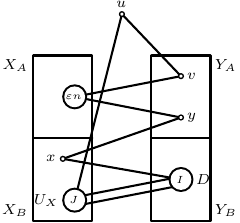}
    \caption{$E(X_A,Y_B)\neq \varnothing$, $|U_X|\geq \frac{n}{6}$.}
  \end{subfigure}
    \begin{subfigure}[b]{.34\linewidth}
    \centering
    \includegraphics[width=0.7\linewidth]{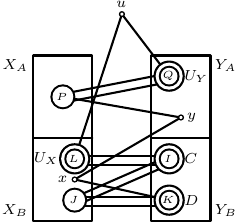}
    \caption{$E(X_A,Y_B)\neq \varnothing$, $|U_X|< \frac{n}{6}$.}
  \end{subfigure}
  \caption{Proof of $e(X_A,Y_B)=0=e(X_B,Y_A)$. }
  \label{fig:thm194-8-10}
\end{figure}

\begin{claim}
$e(X_A,Y_B)=0=e(X_B,Y_A)$.
\end{claim}

\begin{poc}
First assume that there is an edge $xy$ with $x\in X_A$, $y\in Y_B$. Let $v\in U_Y$ and let $C=N(x)\cap Y_A\setminus \{v\}$. 
Since $|X\setminus B|\leq n/2-|B|\leq n/3$, $|N(v)\cap N(w)\cap (X\setminus B)|\geq \varepsilon n$ for every $w\in C$. 
By Theorem \ref{thm:KST}, there exists a copy of $K_{h,h}$  with parts $I\subseteq N(v)\cap (X\setminus B)\setminus \{x\}$ and $J\subseteq C$. 
Since $|U_X|\geq \varepsilon n$, by Lemma \ref{lem-intersection} there is a $K_{h,\gamma n}$ with $L\subseteq U_X$ and $R \subseteq Y_B$. 
Since $|N(y)\cap N(y')\cap (X\setminus A)|\geq \varepsilon n$ for every $y'\in R$, one can find a copy  of $K_{h,h}$ with parts $P\subseteq N(y)\cap (X\setminus A)\setminus L$ and $Q\subseteq R$. Now $(u,v,I,J,x,y,P,Q,L)$ forms a copy of $C_9[1,1,h,h,1,1,h,h,h]$ as shown in Figure \ref{fig:thm194-8-10} (a), which contradicts \textbf{B4}.

Next assume that there is an edge $xy$ with $x\in X_B$, $y\in Y_A$. 
Recall that $|X|\le n/2$ and $|A|\ge n/6$, we have $|X\setminus A|\le n/3$.
Let $D=N(x)\cap Y_B$. Then $|N(y')\cap (X\setminus A)|\geq (1/6+\varepsilon/2) n$ for each $y'\in D$. 
Note that $|A\setminus X_A|<\gamma n$ and $U_X\subseteq X_B$ imply $|U_X \cap (X\setminus A)|\geq |U_X|-\gamma n$
If $|U_X|\geq n/6$, then $|N(y')\cap U_X|\geq \varepsilon n/3$ for each $y'\in D$. 
By Theorem \ref{thm:KST}, there exists a $K_{h,h}$ with parts $I\subseteq D$ and $J\subseteq U_X$. If $v=y$ then $(u,v,x,J,I)$ forms a copy of $C_5[1,1,1,h,h]$, a contradiction. Thus $v\neq y$. 
Since $|X\setminus B|\leq n/3$ implies $|N(v)\cap N(y)|\geq \varepsilon n$, $(u,J,I,x,y,N(x)\cap N(v),v)$ forms a copy of $C_7[1,h,h,1,1,|N(x)\cap N(v)|,1]$ as shown in Figure \ref{fig:thm194-8-10} (b), which contradicts \textbf{B3}. Thus we may assume $|U_X|< n/6$ and thereby $|U_Y|\geq \varepsilon n$.

By $|X\setminus B|<n/3$, we have $|N(y) \cap N(y') \cap (X \setminus B)| \geq 2(1/6 + \varepsilon/2)n - n/3 \geq \varepsilon n$ for each $y' \in U_Y \subseteq Y_A$.
By Theorem \ref{thm:KST}, we obtain a $K_{h,h}$ with parts $P\subseteq N(y)\cap (X\setminus B)$ and $Q\subseteq U_Y\setminus \{y\}$. 
By Lemma  \ref{lem-intersection}, there is $K_{h,\gamma n}$ with parts  $L\subseteq U_X$ and  $C\subseteq Y_B$. Let $D=N(x)\setminus C$. 
Apply the Pairing Lemma (Lemma \ref{lmm:pairing}) to $(C,D,X_B\setminus (L\cup \{x\}))$, we obtain a copy of $P_3[h]$ with blocks $I\subseteq C, J\subseteq X_B\setminus (L\cup \{x\})$ and $K\subseteq D$. 
Then $(u,L,I,J,K,x,y,P,Q)$ contains a copy  of $C_9[1,h,h,h,h,1,1,h,h]$ as shown in Figure \ref{fig:thm194-8-10} (c), which contradicts \textbf{B4}. 
\end{poc}

Thus, $(X_A\cup Y_B\cup \{u\}, X_B\cup Y_A)$ forms a bipartition of $G$.

\begin{mycase}{Case 2.}
    $|U_X|< \varepsilon n$, thus $|U_Y|\ge n/6$.
\end{mycase}

\begin{figure}[ht]
    \centering
    \includegraphics[width=0.24\linewidth]{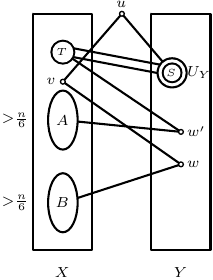}
    \caption{The structure of $A$ and $B$.}
     \label{fig:thm1-11iv-case2}
\end{figure}

Let $t=12h$. Then by Lemma \ref{lem-finalkey}, $G$ is both $\Gamma(t)$-free and $\Gamma'(t)$-free. 
Let $v\in U_X$. Since $|N(y)\cap X|\geq n/6$ for each $y\in U_Y$, by Lemma \ref{lem-intersection}  there is a $K_{t,\gamma n}$ with parts $S\subseteq U_Y$ and $T_0\subseteq X$.  
Applying Lemma \ref{lem-intersection} again we obtain a $K_{t,\gamma n}$ with  $T\subseteq T_0\setminus\{v\}$ and $C\subseteq Y\setminus S$. Let $D\subseteq N(v)\cap Y\setminus (S\cup C)$ with $|D|=\gamma n$. 

We claim that there exist $w\in D$ and $w'\in C$ such that $|N(w)\cap N(w'|<\gamma n$. Indeed, otherwise $|N(w)\cap N(w')|\geq \gamma n$ for any $w\in D$ and $w'\in C$. By Theorem \ref{thm:KST}, one can find a $P_3[h]$ with blocks $I\subseteq C$, $J\subseteq X$ and $K\subseteq D$. Then $(u,S,T,I,J,K,v)$ contains a copy of $C_7[1,h,h,h,h,h,1]$, which contradicts \textbf{B3}. Thus one can choose  $w\in D$ and $w'\in C$ such that $|N(w)\cap N(w')|<\gamma n$. Let $A=N(w')\cap X \setminus (T\cup \{v\})$ and $B=N(w)\setminus (A\cup T\cup \{v\})$. 
Clearly $|A|> n/6$. By $|N(w)\cap N(w')|<\gamma n$, we also have  $|B|>n/6$ as shown in Figure \ref{fig:thm1-11iv-case2}. 

Since $|X| \leq n/2$, we infer that either $|N(y)\cap A|\geq \varepsilon n/2$ or $|N(y)\cap B|\geq \varepsilon n/2$ for each $y\in Y$. Partition $Y$ into $Y_A\cup Y_B\cup Y'$, where
\begin{align*}
    Y_A&=\{y\in Y\colon |N(y)\cap A|\ge \varepsilon n/2 \text{ and }|N(y)\cap B|< \varepsilon n/2\},\\
    Y_B&=\{y\in Y\colon |N(y)\cap B|\ge \varepsilon n/2 \text{ and }|N(y)\cap A|< \varepsilon n/2\},\\
    Y'&=Y\setminus (Y_A\cup Y_B).
\end{align*}

\begin{claim}\label{cl:5.8}
    $|N(v)\setminus Y_B|\leq \varepsilon n$.
\end{claim}

\begin{proof}
Otherwise, let $Y_v\subseteq N(v)\cap (Y\setminus Y_B)$ with $|Y_v|=\varepsilon n$. Then for each $y\in Y_v$, $|N(y)\cap A|>\varepsilon n/2$. Consider the bipartite graph $G[A,Y_v]$. Then there exists $v'\in A$ with
\[
|N(v')\cap Y_v| \geq  \frac{1}{|A|} \sum_{y\in Y_v} |N(y)\cap A|> \frac{1}{n}|Y_v|\frac{\varepsilon n}{2} > \frac{\varepsilon^2n}{2}>2\gamma n. 
\]
Let $I\subseteq N(v')\cap Y_v \setminus (S\cup \{w'\})$ with $|I|=\gamma n$. Then  $(u,S,T,w',v',I,v)$ contains a copy of $C_7[1,h,h,1,1,\gamma n,1]$, which contradicts \textbf{B3}. 
\end{proof}

Moreover, since \(v\in U_X\) and \(\delta(G)\ge (1/6+\varepsilon)n\), Claim \ref{cl:5.8} gives
$|N(v)\cap Y_B|\ge d(v)-|N(v)\setminus Y_B|\ge n/6$.
In particular, after deleting \(o(n)\) or any fixed number of previously chosen vertices,
\(N(v)\cap Y_B\) still contains all subsets of size used below, such as \(\varepsilon n/3\) or \(\varepsilon n\),
provided \(\varepsilon>0\) is chosen sufficiently small and \(n\) is sufficiently large.

\begin{claim}
    $|Y'|< \varepsilon n/4$.
\end{claim}

\begin{poc}
   If $|Y'|\geq \varepsilon n/4$, then by applying Lemma \ref{lem-intersection} twice, there exists a copy of $P_3[t]$ with blocks $I\subseteq A, J\subseteq Y'$ and  $K\subseteq B$. Moreover, since $|N(v)\setminus Y_B|<\varepsilon n$, there exists $V_0\subseteq N(v)\cap Y_B\setminus (S\cup \{w',w\})$ with $|V_0|=\varepsilon n$. For any $y\in V_0$, we have $|N(y)\cap (B\setminus K)|>\varepsilon n/3$. By Theorem \ref{thm:KST}, one can find a $K_{t,t}$ with parts $V\subseteq V_0$ and $W\subseteq B\setminus K$.  Then $(u,v,w,K,J,I,w',T,S)+(V,W)$ contains a copy of $\Gamma(t)$,   contradicting Lemma \ref{lem-finalkey}. 
\end{poc}

Since $|Y'|\leq \varepsilon n/4$ and $\delta(G)\geq (1/6+\varepsilon) n$, we infer that $|N(x)\cap Y_A|\ge \varepsilon n/2$ or $|N(x)\cap Y_B|\ge \varepsilon n/2$ for all $x\in X$. Partition $X$ into $X_A\cup X_B\cup X'$, where
\begin{align*}
    X_A&=\{x\in X: |N(x)\cap Y_A|\ge \varepsilon n/2 \text{ and }|N(x)\cap Y_B|< \varepsilon n/2\},\\
    X_B&=\{x\in X: |N(x)\cap Y_B|\ge \varepsilon n/2 \text{ and }|N(x)\cap Y_A|< \varepsilon n/2\}, \\
    X'&=X\setminus (X_A\cup X_B).
\end{align*}

It is easy  to see that $w'\in Y_A$ and $w\in Y_B$.

\begin{claim}\label{cl:5.9}
 $U_X\subseteq X_B$ and $U_Y\subseteq Y_A\cup Y'$.
\end{claim}

\begin{poc}
If there is some $x\in U_X\setminus X_B$, then $|N(x)\cap Y_A \setminus (S\cup \{w,w'\})| \geq \varepsilon n/3$. Choose $C\subseteq N(x)\cap Y_A \setminus (S\cup \{w,w'\})$ with $|C|  =\varepsilon n/3$. 
Since $|N(y)\cap A|\geq \varepsilon n/2$ for all $y\in C$, by Theorem \ref{thm:KST} we obtain a copy of $K_{2h,2h}$ with parts $I,J$ in $G[A,C]$. Then $(u,x,J,I,w',T\setminus \{x\},S)$ contains a copy of $C_7[1,1,h,h,1,h,h]$, which contradicts \textbf{B3}. Thus $U_X\subseteq X_B$.

Next, assume that there is some $y\in U_Y\cap Y_B$. Since $v\in U_X\subseteq X_B$, one can choose $D\subseteq N(v)\cap Y_B\setminus \{y\}$ with $|D|=\varepsilon n/3$. Since $D\cup \{y\}\subseteq Y_B\cup Y'$, we infer that $|N(y')\cap A|<\varepsilon n/2$ for any $y'\in D\cup \{y\}$. Since $|X\setminus A| <n/3$, we infer that $|N(y)\cap N(y')\cap (X\setminus A)| \geq \varepsilon n$ for any $y'\in D$. 
By Theorem \ref{thm:KST}, there exists a copy of $K_{h,h}$ with parts $P,Q$ in $G[N(y)\cap (X\setminus A),D]$. Then $(P,Q,v,u,y)$ forms a copy of $C_5[h,h,1,1,1]$, which contradicts \textbf{B1}. 
\end{poc}

\begin{figure}[ht]
  \centering
  \begin{subfigure}[b]{.32\linewidth}
  \centering
    \includegraphics[width=0.88\linewidth]{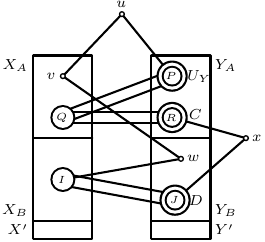}
    \caption{ The case $x\in X'$.}
  \end{subfigure}
  \begin{subfigure}[b]{.32\linewidth}
    \centering
    \includegraphics[width=0.8\linewidth]{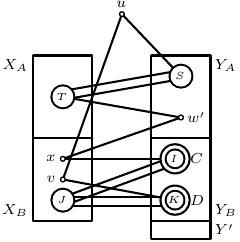}
    \caption{The case  $x\in A\cap  X_B$.}
  \end{subfigure}
    \begin{subfigure}[b]{.32\linewidth}
    \centering
    \includegraphics[width=0.8\linewidth]{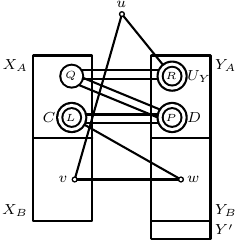}
    \caption{The case  $x\in B\cap X_A$.}
  \end{subfigure}
  \caption{Proof of claims}
  \label{fig:thm1942-2-4}
\end{figure}

\begin{claim}
    $X'=\varnothing$. 
\end{claim}

\begin{poc}
    Suppose that there is some $x\in X'$. Choose $C\subseteq N(x)\cap Y_A$ and $D\subseteq N(x)\cap Y_B\setminus \{w\}$ with $|C|=|D|=\varepsilon n/3$. 
    Note $U_Y\subseteq Y_A\cup Y'$ by Claim \ref{cl:5.9}. Since $|X\setminus B|\leq n/3$ and $|N(y)\cap B|<\varepsilon n/2$ for each $y\in U_Y\cup C$, we infer that $|N(y)\cap (X\setminus B)|\geq (1/2+\varepsilon) |X\setminus B|$. Applying the Pairing Lemma (Lemma \ref{lmm:pairing}) to $(U_Y\setminus C,C, X\setminus B)$, we obtain a copy of $P_3[t]$ with blocks $P\subseteq U_Y\setminus C$, $Q\subseteq X\setminus B$, $R\subseteq C$.  

    Since $D\cup \{w\}\subseteq Y_B$, $|N(y)\cap A|<\varepsilon n/2$ for each $y\in D\cup \{w\}$. Then $|N(y)\cap N(w)\cap (X\setminus A)| \geq \varepsilon n$ for any $y\in D$. By Theorem \ref{thm:KST}, there exists a copy of $K_{t,t}$ with parts $I,J$ in $G[N(w)\cap (X\setminus A), D]$. Then $(u,P,Q,R,x,J,I,w,v)$ contains a copy of $C_9[1,t,t,t,1,t,t,1,1]$ as shown in Figure \ref{fig:thm1942-2-4} (a). 
    
    Moreover, since $|N(v)\setminus Y_B|<\varepsilon n$, there exists $V_0\subseteq N(v)\cap Y_B\setminus (P\cup R\cup J\cup \{w\})$ with $|V_0|=\varepsilon n$. For any $y\in V_0$, we have $|N(y)\cap (B\setminus (I\cup Q\cup\{v\}))|>\varepsilon n/3$. By Theorem \ref{thm:KST}, one can find a $K_{t,t}$ with parts $V\subseteq V_0$ and $W\subseteq B\setminus (I\cup Q\cup\{v\})$.  Then $(u,P,Q,R,x,J,I,w,v)+(w,v,V,W)$ forms a copy of $\Gamma'(t)$, contradicting Lemma \ref{lem-finalkey}.  
\end{poc}

\begin{claim}
    $A\subseteq X_A$ and $|B\cap X_A|\leq \varepsilon n/4$.
\end{claim}

\begin{poc}
Suppose that there is some $x\in A\cap X_B$. 
Note that $v\notin A$, for otherwise $(u,v,w',T,S)$ would forms a copy of $C_5[1,1,1,t,t]$, a contradiction. 
Since $x,v\in X_B$, one can choose $C\subseteq N(x)\cap Y_B$ and $D\subseteq N(v)\cap Y_B\setminus C$ with $|C|=|D|=\varepsilon n/3$. 
Since $|X\setminus A|\leq  \frac{n}{2}-|A| \leq \frac{n}{3}$, we infer that $|N(y)\cap N(y')\cap (X\setminus (A\cup T\cup \{x,v\}))|\geq \varepsilon n$ for any $y\in C$ and $y'\in D$. 
Then apply the Pairing Lemma (Lemma \ref{lmm:pairing}) to $(C,D,X\setminus (A\cup T\cup \{x,v\}))$, there exists a copy of $P_3[h]$ with blocks $I\subseteq C,J\subseteq X\setminus (A\cup T\cup \{x,v\}),K\subseteq D$. 
Now $(K,J,I,x,w',T,S,u,v)$ contains a copy of $C_9[h,h,h,1,1,h,h,1,1]$ as shown in Figure \ref{fig:thm1942-2-4} (b), which contradicts \textbf{B4}. Thus $A\subseteq X_A$.

Next, assume that $|B\cap X_A|>\varepsilon n/4$. Choose $C\subseteq B\cap X_A$ with $|C|=\varepsilon n/4$. 
By Lemma \ref{lem-intersection}, there is a $K_{h,\gamma n}$ with parts  $L\subseteq C$ and $D\subseteq Y_A$. 
Since $|X\setminus B|\leq n/3$, apply the Pairing Lemma (Lemma \ref{lmm:pairing}) to $(U_Y\setminus D, D,X\setminus (B\cup L))$, there exists a copy of $P_3[2h]$ with blocks $P\subseteq D$, $Q\subseteq X\setminus (B\cup L)$, $R\subseteq U_Y\setminus D$. Then $(u,v,w,L,P,Q,R)$ forms a copy of $C_7[1,1,1,h,h,h,h]$ as shown in Figure \ref{fig:thm1942-2-4} (c), which contradicts \textbf{B1}.     
\end{poc}

\begin{figure}[ht]
  \centering
  \begin{subfigure}[b]{.4\linewidth}
  \centering
    \includegraphics[width=0.7\linewidth]{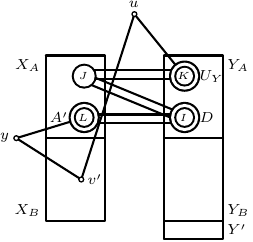}
    \caption{ The case $v'\in B'\cap U_X$.}
  \end{subfigure}
  \begin{subfigure}[b]{.4\linewidth}
    \centering
    \includegraphics[width=0.7\linewidth]{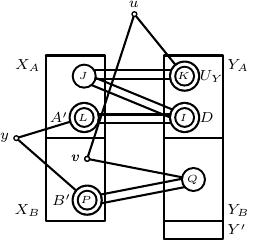}
    \caption{The case  $|Y_B|\leq \frac{n}{3}$.}
  \end{subfigure}
    \begin{subfigure}[b]{.4\linewidth}
    \centering
    \includegraphics[width=0.7\linewidth]{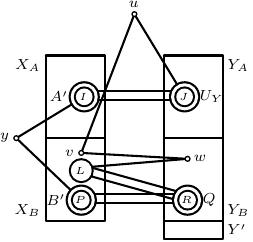}
    \caption{The case  $|Y_A|\leq |U_Y|+\frac{n}{6}$.}
  \end{subfigure}
      \begin{subfigure}[b]{.4\linewidth}
    \centering
    \includegraphics[width=0.7\linewidth]{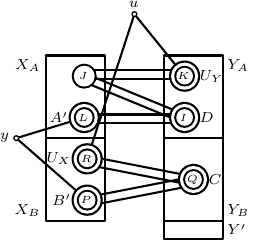}
    \caption{The fourth case.}
  \end{subfigure}
  \caption{Proof of  $Y'=\varnothing$.}
  \label{fig:thm1942-5-8}
\end{figure}

\begin{claim}
    $Y'=\varnothing$. 
\end{claim}

\begin{poc}
Assume $y\in Y'$. Then there exist $A'\subseteq N(y)\cap X_A$ and $B'\subseteq N(y)\cap X_B \setminus \{v\}$ with $|A'|=|B'|=\varepsilon n/4$. 
By Lemma \ref{lem-intersection}, one can find a $K_{h,\gamma n}$ with parts $L\subseteq A'$ and $D\subseteq Y_A$. Note that $|X\setminus B|\leq n/3$. 
Then apply the Pairing Lemma (Lemma \ref{lmm:pairing}) to $(D,U_Y\setminus D, X\setminus (B\cup L))$, we obtain a copy of $P_3[h]$ with blocks $I\subseteq D, J\subseteq X\setminus (B\cup L), K\subseteq U_Y\setminus D$. 

If  $B'\cap U_X\neq \varnothing$, then fix a vertex $v'\in B'\cap U_X$.
Note $(u,v',y,L,I,J,K)$ forms a copy of $C_7[1,1,1,h,h,h,h]$ as shown in Figure \ref{fig:thm1942-5-8} (a), which contradicts \textbf{B1}. 
Thus, we may assume $B'\cap U_X=\varnothing$.

If $|Y_B|\leq n/3$, then $|N(x)\cap N(v)\cap Y_B|\geq 2(1/6+\varepsilon/4)n-n/3 =\varepsilon n/2$ for each $x\in B'$. 
By Theorem \ref{thm:KST}, we obtain a $K_{h,h}$ with parts $P\subseteq B'$ and $Q\subseteq N(v)\cap Y_B$. 
It follows that $(L,I,J,K,u,v,Q,P,y)$ forms a copy of $C_9[h,h,h,h,1,1,h,h,1]$  as shown in Figure \ref{fig:thm1942-5-8} (b), which contradicts \textbf{B4}. Thus  we may assume $|Y_B|>n/3$.

Note that $|U_Y|\geq n/6$ and $|N(x)\cap Y_A|\geq (1/6+\varepsilon/4)n$ for any $x\in A'$. 
If $|Y_A|\leq |U_Y|+n/6$, then $|N(x)\cap U_Y|\ge|N(x)\cap Y_A|+|U_Y|-|Y_A|\geq \varepsilon n/4$ for any $x\in A'$. 
By Theorem \ref{thm:KST}, we obtain a $K_{t,t}$ with parts $I\subseteq A'$ and $J\subseteq U_Y$. By Lemma \ref{lem-intersection}, there exist $P\subseteq B'$ and $Q\subseteq Y_B\setminus \{w\}$ such that $G[P,Q]$ forms a $K_{t,\gamma n}$. Since  for any $w'\in Q$,
\[
|N(w)\cap N(w')\cap (X\setminus A)| \geq 2\left(1/6+\varepsilon/2\right)n -|X\setminus A| >\varepsilon n,
\]
by Theorem \ref{thm:KST} there is a $K_{t,t}$ with parts $L\subseteq N(w)\cap (X\setminus A)\setminus (P\cup I \cup \{v\})$ and $R\subseteq Q$. Then $(u,v,w,L,R,P,y,I,J)$ forms a copy of $C_9[1,1,1,t,t,t,1,t,t]$, as shown in Figure \ref{fig:thm1942-5-8} (c).  
Moreover, since $|N(v)\setminus Y_B|<\varepsilon n$, there exists $V_0\subseteq N(v)\cap Y_B\setminus (R\cup \{w\})$ with $|V_0|=\varepsilon n$. For any $y\in V_0$, we have $|N(y)\cap (B\setminus (P\cup L\cup\{v\}))|>\varepsilon n/3$. By Theorem \ref{thm:KST}, one can find a $K_{t,t}$ with parts $V\subseteq V_0$ and $W\subseteq B\setminus (P\cup L\cup\{v\})$.  Then $(u,v,w,L,R,P,y,I,J)+(w,v,V,W)$ forms a copy of $\Gamma(t)$, contradicting Lemma \ref{lem-finalkey}.  Thus we may assume $|Y_A|> |U_Y|+n/6$. 

Now $|Y|=|Y_A|+|Y_B|\geq n/2+|U_Y|$ and  thereby $|X|\leq n/2-|U_Y|$. Since $|A|\geq n/6$ and $|U_X|+|U_Y|> n/6$, we have
\[
|X\setminus A|\leq n/2-|U_Y|-|A|\leq n/6+|U_X|.
\]
Note that $|N(y')\cap A|<\varepsilon n/2$ for any $y'\in Y_B$. 
It follows that for any $y'\in Y_B$, \[
|N(y')\cap U_X|\ge |N(y')\cap(X\setminus S)|+|U_X|-|X\setminus A|\geq \varepsilon n/2.
\]
Thus $|U_X|\geq \varepsilon n/2$.  
By Lemma \ref{lem-intersection}, there exist $P\subseteq B'$ and $C\subseteq Y_B$ such that $G[P,C]$ forms a $K_{h,\gamma n}$. 
Since $|N(y')\cap U_X|\geq \varepsilon n/2$ for any $y'\in C$,  by Theorem \ref{thm:KST}, there exists a $K_{h,h}$ with parts $R\subseteq U_X$ and $Q\subseteq C$. It follows that $(L,I,J,K,u,R,Q,P,y)$ forms a copy of $C_9[h,h,h,h,1,h,h,h,1]$  as shown in Figure \ref{fig:thm1942-5-8} (d), which contradicts \textbf{B4}. 
\end{poc}

\begin{figure}[ht]
  \centering
  \begin{subfigure}[b]{.4\linewidth}
  \centering
    \includegraphics[width=0.7\linewidth]{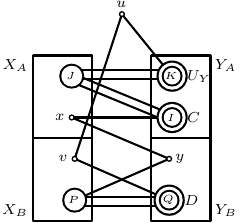}
    \caption{ The case $xy\in E(X_A,Y_B)$.}
  \end{subfigure}
  \begin{subfigure}[b]{.4\linewidth}
    \centering
    \includegraphics[width=0.7\linewidth]{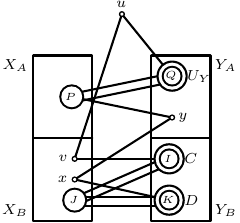}
    \caption{The case $xy\in E(X_B,Y_A)$.}
  \end{subfigure}
  \caption{Proof of  $e(X_A,Y_B)=0=e(X_B,Y_A)$.}
  \label{fig:thm1942-9-10}
\end{figure}

\begin{claim}
    $e(X_A,Y_B)=0=e(X_B,Y_A)$. 
\end{claim}

\begin{poc}
First, assume  that there is an edge $xy$ with $x\in X_A$ and $y\in Y_B$. Choose $C\subseteq N(x)\cap Y_A$ with $|C|=\varepsilon n$. 
Apply the Pairing Lemma (Lemma \ref{lmm:pairing}) to $(C,U_Y\setminus C,X\setminus(B\cup \{x\}))$, there exists a copy $P_3[h]$ with $I\subseteq C$, $J\subseteq X\setminus(B\cup \{x\})$ and $K\subseteq U_Y\setminus C$. 
On the other hand, let \(D\subseteq N(v)\cap Y_B\setminus\{y\}\) with \(|D|=\varepsilon n\). Since $|N(y)\cap N(y')\cap (X\setminus A)|\geq \varepsilon n$ for any $y'\in D$, by Theorem \ref{thm:KST}, we obtain a $K_{h,h}$ with parts $P\subseteq N(y)\cap \bigl(X\setminus (A\cup J\cup\{x,v\})\bigr)$, and $Q\subseteq D$. 
Now $(u,K,J,I,x,y,P,Q,v)$ forms a copy of $C_9[1,h,h,h,1,1,h,h,1]$ as shown in Figure \ref{fig:thm1942-9-10} (a), which contradicts \textbf{B4}. Thus $e(X_A,Y_B)=0$. 

Next, assume that there is an edge $xy$ with $x\in X_B$ and $y\in Y_A$.  Note that $|N(y)\cap N(y')\cap (X\setminus B)| \geq \varepsilon n$ for all $y'\in U_Y\setminus \{y\}$.  By Theorem \ref{thm:KST}, there is a $K_{h,h}$ with parts $P\subseteq N(y)\cap (X\setminus B)$ and $Q\subseteq U_Y\setminus \{y\}$. 
Let $C\subseteq N(v) \cap Y_B$ and $D\subseteq N(x)\cap Y_B\setminus C$ with $|C|=|D|=\varepsilon n$.
Then apply the Pairing Lemma (Lemma \ref{lmm:pairing}) to $(C,D,X\setminus (A\cup \{v,x\}))$, there exists a copy of $P_3[h]$ with blocks $I\subseteq C,J\subseteq X\setminus (A\cup P\cup \{v,x\}), K\subseteq D$. 
Then $(u,Q,P,y,x,K,J,I,v)$ forms a copy of $C_9[1,h,h,1,1,h,h,h,1]$ as shown in Figure \ref{fig:thm1942-9-10} (b), which contradicts \textbf{B4}.     
\end{poc}

Now $(X_A\cup Y_B\cup \{u\}, X_B\cup Y_A)$ forms a bipartition of $G$ and the lemma is proven.
\end{proof}

\section{Connections to regular Tur\'an numbers}\label{sec:regular}

In this section, we highlight an application of our main result: the determination of certain regular Tur\'an numbers. The \emph{regular Tur\'an number} of a graph $H$ is defined as
\[
\mathrm{regex}(n,H)=\max\{d:\text{there exists an $n$-vertex $d$-regular $H$-free graph}\}.
\]
Introduced by Caro and Tuza~\cite{caro2020regular} and independently studied by Cambie, de Verclos, and Kang~\cite{cambie2023regular}, this parameter combines classical Tur\'an-type problems with the additional constraint of regularity.
For graphs with chromatic number at least $4$, regular Tur\'an numbers are already well-understood:

\begin{theorem}[\cite{cambie2023regular}]
    Let $H$ be a graph with $\chi(H) \ge 4$. Then
\[
\mathrm{regex}(n,H) = \Big( 1 - \frac{1}{\chi(H) - 1} + o(1) \Big)n.
\]
\end{theorem}

For $3$-chromatic graphs, the situation is more delicate. Cambie, de Verclos, and Kang obtained partial results, particularly for certain tripartite graphs. Denote by $K_{2x,y}^{=}$ the complete bipartite graph $K_{2x,y}$ with a perfect matching on the part of size $2x$.

\begin{theorem}[\cite{cambie2023regular}]
Let $H$ be a graph with $\chi(H) = 3$.

\begin{enumerate}
    \item Suppose one of the following holds:
    \begin{itemize}
        \item for every vertex $v$ of $H$, the graph $H \setminus v$ is not bipartite; or
        \item $H$ is not a subgraph of $K_{2|H|,|H|}^{=}$.
    \end{itemize}
    Then $\mathrm{regex}(n,H) = (1/2 + o(1))n.$
    \item If neither of the above hold and $n$ is odd, then
 $\mathrm{regex}(n,H) \le 2 \left\lfloor \frac{n}{5} \right\rfloor$.
\end{enumerate}
\end{theorem}

While these results cover several cases, many $3$-chromatic graphs remained unresolved. Our main theorem allows us to close this gap for a natural family of $3$-chromatic graphs, giving a precise asymptotic value for their regular Tur\'an numbers.

\begin{theorem}\label{thm:reg}
    Let $H$ be a graph with $\chi(H)=3$ and $\delta_{\chi}(H,2)\in\big\{\frac{2}{5},\frac{2}{7},\frac{2}{9},\frac{2}{11}\big\}$.
    Then for odd $n$, $\mathrm{regex}(n, H)=(\delta_{\chi}(H,2)+o(1))n$.
\end{theorem}

Recall that $\mathcal{C}_{2k-1}=\{C_3,\ldots,C_{2k-1}\}$.
Cambie, de Verclos, and Kang \cite{cambie2023regular} constructed a family of graphs with large odd girth.
\begin{lemma}[\cite{cambie2023regular}]\label{cambie2023regular}
    For sufficiently large odd $n$, there exists an $n$-vertex $\mathcal{C}_{2k-1}$-free graph which is $(\frac{2}{2k+1}-o(1))n$-regular.
\end{lemma}

We now combine this construction with Theorem~\ref{thm-main} to determine the regular Tur\'an number asymptotically for several graphs $H$ with chromatic number $3$.

\begin{proof}[Proof of Theorem \ref{thm:reg}]
    Suppose that $\delta_{\chi}(H,2)=\frac{2}{2k+1}$ for some $2\le k\le 5$.
    Then by Theorem \ref{thm-main}, $H$ has odd girth $2k-1$.
    Let $G$ be an $n$-vertex $d$-regular $H$-free graph, where $n$ is odd.
    Since $n$ is odd, $G$ cannot be bipartite, because in every regular bipartite graph the two parts have the same size. 
    By Theorem \ref{thm-main}, $d=\delta(G)\le (\frac{2}{2k+1}+o(1))n$.
    On the other hand, by Lemma \ref{cambie2023regular}, 
    for sufficiently large odd $n$, there exists an $n$-vertex $H$-free graph which is $(\frac{2}{2k+1}-o(1))n$-regular.
    Thus $\mathrm{regex}(n, H)\ge(\frac{2}{2k+1}-o(1))n$.
    Combining the upper and lower bounds, we obtain $\mathrm{regex}(n, H)=(\delta_{\chi}(H,2)+o(1))n$.
\end{proof}

\section{Concluding remarks}\label{sec:concluding}
In this paper, we determine all possible values of $\delta_{\chi}(H,2)$ for graphs $H$ with $\chi(H)=3$.
This naturally leads to the following problem.

\begin{problem}
    Determine all possible values of $\delta_{\chi}(H,3)$ for graphs $H$ with $\chi(H)=4$.
\end{problem}

We believe that studying the vertex-extendable threshold provides a useful approach to this problem.
We prove the following general bound for color-critical graphs.

\begin{theorem}\label{thm:color-critical}
    Let $H$ be a color-critical graph with $\chi(H)=r+1\geq 3$. Then 
\[
\delta_{\mathrm{ext}}(H,r) \le \frac{3r-4}{3r-1}. 
\]
\end{theorem}

To show that the bound in Theorem~\ref{thm-AES} is tight, Andr\'{a}sfai, Erd\H{o}s, and S\'{o}s \cite{AES74} constructed a $K_{r+1}$-free graph with minimum degree at least $\frac{3r-4}{3r-1}n$ that is not $r$-colorable.
Let $\ell=\lfloor\frac{n}{3r-1}\rfloor$. Define $W_{n,r}$ as the complete join of a balanced blow-up of $C_5$ and a Tur\'{a}n graph $T_{r-2}(n-5\ell)$, denoted by $W_{n,r} := C_5[\ell] \vee T_{r-2}(n-5\ell)$.

  Let $H$ be a color-critical graph with $\chi(H)=r+1\geq 3$. 
  Then $H$ is a subgraph of a blow-up of $K_{r+1}$. 
   Combining Theorem \ref{thm-AES}, Theorem \ref{thm:color-critical}, and Lemma \ref{lmm:upper-bound}, we obtain $\delta_{\chi}(H,r)\le \max\{\delta_{\chi}(K_{r+1},r),\delta_{\mathrm{ext}}(H,r)\}= \frac{3r-4}{3r-1}$.

\begin{corollary}
    Let $H$ be a color-critical graph with $\chi(H)=r+1\geq 3$. If $H\not\subseteq W_{n,r}$, then $\delta_{\chi}(H,r) = \frac{3r-4}{3r-1}.$
\end{corollary}

\begin{proof}[\textbf{Proof of Theorem  \ref{thm:color-critical}}]
    Let $G$ be an $H$-free graph with $\delta(G)\geq \big(\frac{3r-4}{3r-1}+\varepsilon\big)n$. 
    Suppose that $G-u$ is $r$-colorable. We will show that $G$ is $r$-colorable.
    Let $V(G)=X_1\cup\cdots\cup X_r$ be the corresponding $r$-partition of $G-u$.
    Suppose for the sake of contradiction that $G$ is not $r$-colorable. Then $u$ must have neighbors in all $X_i$.
    Let $A_i=N(u)\cap X_i$ for $i\in[r]$.
    Without loss of generality, assume that $1\le |A_1|\le \cdots\le |A_r|$.

\begin{claim}\label{cl:Xi}
   For $1\le i\le r$, $\frac{2}{3r-1}n\le |X_i|\le \frac{3}{3r-1}n$.
\end{claim}

\begin{poc}
    For any $x \in X_i$, $N(x) \setminus \{u\}\subseteq V(G) \setminus X_i$. Thus, $\delta(G) \le d(x) \le n - |X_i|$, implying
$|X_i| \le \frac{3}{3r-1}n$.
For the lower bound, $|X_i| = (n-1) - \sum_{j \neq i} |X_j| \ge n - \frac{3(r-1)}{3r-1}n  = \frac{2}{3r-1}n$.
\end{poc}

\begin{claim}\label{cl:Ai}
    For $2\le i\le r$, $|A_i|\ge \frac{1}{i}\big(\frac{3i-4}{3r-1}+\varepsilon\big)n$.
\end{claim}

\begin{poc}
    We use the minimum degree condition on $u$. We have $d(u) = \sum_{j=1}^r |A_j| \ge \big(\frac{3r-4}{3r-1}+\varepsilon\big)n$.
    Recall the assumption $|A_1| \le \dots \le |A_r|$. We have
    \[
    \left(\frac{3r-4}{3r-1}+\varepsilon\right)n \le \sum_{j=1}^r |A_j| \le i|A_i| + \sum_{j=i+1}^r |X_j| \le i|A_i| + (r-i)\frac{3}{3r-1}n.
    \]
    Rearranging terms yields:
    \begin{align*}
        i|A_i| &\ge \left(\frac{3r-4}{3r-1} - \frac{3r-3i}{3r-1} + \varepsilon\right)n = \left(\frac{3i-4}{3r-1} + \varepsilon\right)n.
    \end{align*}
\end{poc}
Fix a vertex $v\in A_1$. We now estimate the number of copies of $K_{r-1}$ in $G[N(u)\cap N(v)]$.
    We construct such cliques greedily. 
    Let \(z_1=v\). We want to find \((z_2,\ldots,z_r)\) such that
\(z_j\in A_j\) for every \(2\le j\le r\), and
\(\{u,z_1,\ldots,z_r\}\) forms a clique.
Suppose that \(2\le i\le r\) and that we have already chosen a clique
\(\{u,z_1,\ldots,z_{i-1}\}\) with \(z_j\in A_j\) for every \(j<i\).
The number of choices for \(z_i\in A_i\) is
$|\bigcap_{j=1}^{i-1} N(z_j)\cap A_i|$.
    By Claims \ref{cl:Xi} and \ref{cl:Ai}, we have 
\begin{align*}
    \left|\bigcap_{j=1}^{i-1} N(z_j)\cap A_i\right|
\ge &
\left|\bigcap_{j=1}^{i-1} N(z_j)\right|
-
\left|\bigcup_{j=i+1}^r X_j\right|
+ |A_i|-|X_i|.\\
    \ge &\left(\frac{3r-4}{3r-1}+\varepsilon-\frac{3(i-1)}{3r-1}-\frac{3(r-i)}{3r-1}+\frac{3i-4}{i(3r-1)}\right)n\\
    =&\left(\frac{2(i-2)}{i(3r-1)}+\varepsilon\right)n
\end{align*}

Hence, the number of $K_{r-1}$'s in $G[N(u)\cap N(v)]$ is at least $\prod_{i=2}^r\big(\frac{2(i-2)}{i(3r-1)}+\varepsilon\big)n\ge \varepsilon^{r-1}n^{r-1}$.
By Theorem \ref{thm:AS16}, there exists a copy of $K_{r-1}[2rh]$ in $G[N(u)\cap N(v)]$. 
    Decompose this $K_{r-1}[2rh]$ into $2rh$ vertex-disjoint $(r-1)$-cliques, labeled $Q_1, \ldots, Q_{2rh}$.
    Consider an auxiliary bipartite graph $B$ with bipartition $X\cup Y$ as follows: 
    $X=\{Q_1,\ldots,Q_{2rh}\}$ and $Y=V(G)\setminus (V(K_{r-1}[2rh])\cup \{u,v\})$.
    An edge exists between $Q_i$ and $y\in Y$ if $y$ is adjacent to every vertex in $Q_i$.

    Note that $|\bigcap_{w\in V(Q_i)} N(w)|\ge \big(\frac{3r-4}{3r-1}+\varepsilon-\frac{3(r-2)}{3r-1}\big)n=\big(\frac{2}{3r-1}+\varepsilon\big)n$ and $\varepsilon n>2rh$.
    We have $d_Y(x)\ge \frac{2n}{3r-1}$ for every $x\in X$.
    Thus $e(B)\ge \frac{4rhn}{3r-1}>z(2rh,n;h,h-2)$ for sufficiently large $n$.
    By Theorem \ref{thm:KST}, there is a copy of $K_{h,h-2}$ in $B$.
        This structure corresponds to:
    \begin{itemize}
        \item $h$ cliques $Q_{i_1}, \dots, Q_{i_h}$ (forming a $K_{r-1}[h]$ in $G$).
        \item $h-2$ vertices in $Y$, each complete to all $K_{r-1}[h]$.
    \end{itemize}
    Together with $u$ and $v$, this yields a copy of $K_{r}^+[h]$, where an edge is added within an independent set of $K_{r}[h]$, which contains $H$ as a subgraph, a contradiction.
    Thus $G$ is $r$-colorable.
\end{proof}

Finally, we mention a related question for digraphs.
The \emph{chromatic profile of a digraph $H$} is defined as:
\begin{align*}
\delta^+_\chi(H, k):=\inf \{d: \delta^+(D) \geq d|D| \text { and } H \not\subseteq D \Rightarrow \chi(D) \leq k\}.
\end{align*}
Here, $\delta^{+}(D)$ denotes the minimum out-degree of the digraph $D$. 
The digraphs may contain anti-parallel arcs (i.e., directed 2-cycles), but have no loops or parallel arcs. 

Xue \cite{xue2025directed} proved a directed analogue of the Andr\'asfai-Erd\H{o}s-S\'os theorem, stating that $\delta_\chi^{+}(T_r, r-1)=\frac{3 r-7}{3 r-4}$, where $T_r$ is the transitive tournament on $r$ vertices.
Moreover, Xue proved that \(\delta_\chi^{+}(H,2)=1/2\) when \(H\) is a directed 5-cycle, whereas \(\delta_\chi^{+}(H,2)=1/3\) when \(H\) is an oriented 5-cycle that is not directed.
These results motivate the following problem.

\begin{problem}
    Determine $\delta_\chi^{+}(H,2)$ for every orientation of every odd cycle.
\end{problem}

\section*{Acknowledgment}

The project was initiated when the first and third authors visited the second author at Taiyuan University of Technology. The authors would like to thank Professor Weihua Yang for his hospitality during their stay at Taiyuan University of Technology.

\bibliographystyle{abbrv}
\bibliography{sample.bib}

\appendix
\section{Appendix: Proof of Lemma \ref{lem-intersection}}

\begin{proof}
Let $S$ be a random subset of $A$ of size $t$, chosen uniformly at random from all $\binom{|A|}{t}$ possible subsets. Let $Y$ be the random variable which denotes the size of the common neighborhood of $S$, i.e., $$Y = |N(S)| = \left|\bigcap_{x \in S} N(x)\right|.$$

We can express $Y$ as the sum of indicator variables $Y = \sum_{v \in B} I_v$, where $I_v$ is the indicator variable for the event that $v$ is a common neighbor of all vertices in $S$, which is equivalent to $S \subseteq N(v)$. By the linearity of expectation, we have
\begin{align}\label{eq:EY}
    \mathbb{E}[Y] = \sum_{v \in B} \mathbb{E}[I_v] = \sum_{v \in B} \mathbb{P}(S \subseteq N(v)).
\end{align}
Since $S$ is chosen uniformly at random, the probability that $S$ is a subset of the neighborhood of a fixed vertex $v$ is determined by the degree of $v$:
\begin{align*}
    \mathbb{P}(S \subseteq N(v)) = \frac{\binom{d(v)}{t}}{\binom{|A|}{t}}.
\end{align*}
Substituting this into \eqref{eq:EY} yields
\begin{align*}
    \mathbb{E}[Y] = \frac{1}{\binom{|A|}{t}} \sum_{v \in B} \binom{d(v)}{t}.
\end{align*}
Since $d(v) \ge \varepsilon n$ for each $v\in A$, we have
\begin{align*}
    \sum_{v \in B} d(v) = \sum_{v \in A} d(v) \ge |A| \cdot \varepsilon n.
\end{align*}
Let $\bar{d}_B = \frac{1}{|B|}\sum_{v \in B} d(v) \ge \varepsilon |A|$. 
Let $f(z) = \binom{z}{t}=\frac{z(z-1)\cdots(z-t+1)}{t!}$ for $z \ge t$ and $f(z) = 0$ for $0 \le z < t$. 
Note that $f(z)$ is convex and non-decreasing on $[0, \infty)$. Applying Jensen's Inequality, we obtain
\begin{align*}
    \sum_{v \in B} \binom{d(v)}{t} \ge |B| \binom{\bar{d}_B}{t} \ge |B|\binom{\varepsilon |A|}{t}.
\end{align*}
Thus, the expected size of the common neighborhood satisfies
\begin{align*}
    \mathbb{E}[Y] \ge |B|\cdot \frac{\binom{\varepsilon |A|}{t}}{\binom{|A|}{t}}\geq \varepsilon n \frac{\binom{\varepsilon |A|}{t}}{\binom{|A|}{t}}.
\end{align*}
By elementary calculations, for $|A| \ge t/\varepsilon$, we have $\binom{\varepsilon |A|}{t} / \binom{|A|}{t} \ge (\varepsilon/e)^t$.
Therefore there exists a $t$-subset $S\subseteq A$ such that $|N(S)|\ge \gamma n$, and so $G[A,B]$ contains a copy of $K_{t,\gamma n}$. This completes the proof.
\end{proof}
\end{document}